\documentclass{article}
\usepackage{graphicx,amsthm,epsfig,xy,aliascnt,latexsym,float} 
\usepackage{stmaryrd,amsmath}
\usepackage{amsfonts,amssymb,mathrsfs}

\newtheorem{definition}{Definition}

\newtheorem{lemma}{Lemma}
\newtheorem{theorem}{Theorem}

\newcommand{\comment}[1]              {}

\newcommand{\paren}[1]        {\left( #1 \right)}

\newcommand{\ang}[1]         {\langle #1 \rangle}

\newcommand{\mat}[1] {\begin{pmatrix} #1 \end{pmatrix}}

\newcommand         {\n}        {\not}
\newcommand         {\nd}       {\noindent}

\newcommand         {\wt}       {\widetilde}

\newcommand         {\mr}       {\mathrm}

\newcommand         {\mc}       {\mathcal}

\newcommand         {\g}        {\gamma}

\renewcommand       {\l}        {\lambda}

\newcommand         {\0}        {\varnothing}

\newcommand         {\N}        {\mathbb N}
\newcommand         {\Q}        {\mathbb Q}
\newcommand         {\R}        {\mathbb R}

\newcommand         {\C}        {\mathbb C}

\DeclareMathOperator{\y}        {y}

\newcommand         {\inter}    {\cap}
\newcommand         {\Union}    {\bigcup}

\newcommand         {\less}     {\!\setminus\!}

\renewcommand       {\d}        {\partial}

\renewcommand       {\-}        {^{-1}}

\newcommand         {\x}        {\times}

\DeclareMathOperator{\lcm}      {lcm }

\DeclareMathOperator{\tr}       {tr }

\DeclareMathOperator{\MCG}      {MCG}

\DeclareMathOperator{\PSL}      {PSL}

\newcommand         {\inc}      {\hookrightarrow}

\newcommand         {\iso}      {\cong}
\newcommand         {\homeo}    {\approx}

\pdfoutput=1
\input xy
\xyoption{all}

\author{Jeffrey D. Carlson}
\title{Commensurability of two-multitwist pseudo-Anosovs}
\date{\today}

\begin{document}

\maketitle

\begin{abstract}
This paper analyzes commensurability of the class of surface automorphism 
generated by two Dehn multitwists. We show pairwise noncommensurability 
between several classes arising from canonical curve configurations. 
In addition, we consider the Kenyon--Smillie 
invariant $J$ of flat surfaces in this setting. 
We also introduce a general construction of infinite 
classes of commensurable pseudo-Anosov homeomorphisms.
\end{abstract}

\section{Introduction}

Danny Calegari, Hongbin Sun, and Shicheng Wang, in \cite{CSW},
introduced a relation called \emph{commensurability} among surface 
automorphisms. 
An mapping class $\wt \phi \in \MCG(\wt S)$ \emph{covers} 
$\phi \in \MCG(S)$ if
we have a covering $p\colon \wt S \to S$ such that 
$p \wt \phi = \phi p$,
and two mapping classes $\phi_1 \in \MCG(S_1)$ and $\phi_2 \in \MCG(S_2)$ 
are said to be \emph{commensurable} if there are nonzero numbers $m,n \in \mathbb{Z}$
such that $\phi_1^m$ and $\phi_2^n$ are both covered by some map $\wt \phi$. 
In Section 2, we deal with generalities about this relation.

In \cite{Thur}, Bill Thurston introduces a construction for a class of 
pseudo-Anosovs we call {\em two-multiwist pseudo-Anosovs}. 
He associates to each pair of filling multicurves in a surface 
the subgroup of the mapping class group generated by Dehn twists 
around the multicurves, and shows that ``most of'' the elements of 
this group are pseudo-Anosov. This construction 
is appealing in its simplicity and allows some invariants of pseudo-Anosovs 
to be calculated more easily than in the general case. 
We review this construction in Section 3.

In the following Section 4, we review an invariant $J$ of flat structures introduced 
by Richard Kenyon and John Smillie in \cite{KS}. Since the Thurston construction
allows us to determine a flat structure preserved by a collection of 
pseudo-Anosovs, within this class of pseudo-Anosovs
$\R^+ J$ gives a useful commensurability invariant.

Chris Leininger, in \cite{Lein}, studies the Thurston construction 
to understand for which Teichm\"{u}ller curves 
the associated stabilizers contain with finite index 
a group generated by two positive multi-twists. 
In Section 5 we look at the invariants of some of these groups 
determining for some pairs $G,H$ of groups that
no element of $G$ is commensurable with an element of $H$.

Finally, in Section 6, we introduce a simple construction 
that produces many commensurabilities between different elements of different 
two-multitwist groups. The idea is to find covers that either lift 
or uniformly break multicurves, and to consider Dehn twists in the covers.

\subsection{Acknowledgements}
The author would like to thank his advisor 
Genevieve Walsh for asking the question that lead to this paper,
for endless patience in listening to him discuss it, 
and for reading over early versions of it.
The author would also like to thank Dan Margalit for leading him to the paper 
\cite{Lein}, and Curt McMullen for leading him to the paper \cite{KS}. 
Finally the author thanks Aaron W. Brown for his discussions about dynamics and
about $J$.

\section{Generalities}\label{begin}

If $S$ is a compact surface, we will refer to the isotopy class $\phi = [f]$ 
of a homeomorphism $f$ of $S$ as an \emph{automorphism} of $S$. 
We will also refer to an automorphism by the pair $(S,\phi)$.

\begin{definition}
An automorphism $(\wt S,\wt \phi)$ {\em covers} $(S,\phi)$ 
if there are a finite covering $p\colon \wt S \to S$ and $\wt \phi$ and $\phi$ 
have representatives $\wt f$ and $f$ such that $p \circ \wt f = f \circ p 
\colon \wt S \to S$.
We then write $p \colon \wt \phi \to \phi$ or $p_\#(\wt \phi) = \phi$.
\end{definition}

\begin{definition}
Automorphisms  $(S_1,\phi_1)$ and $(S_2,\phi_2)$ are {\em commensurable}
if there are an automorphism $(\wt S,\wt \phi)$ 
and nonzero integers $m,n \in \mathbb{Z}$ such that $(\wt S, \wt \phi)$ covers
$(S_1, \phi_1^m)$ and $(S_2,\phi_2^n)$. 
If $|m| = |n| = 1$ we say $(S_1,\phi_1)$ and $(S_2,\phi_2)$ are 
{\em topologically commensurable}
and if $\wt S = S_1 = S_2$ 
we say they are {\em dynamically commensurable}.
\end{definition}

The relation of commensurability is of interest to 3-manifold theorists because 
if $(S_1,\phi_1)$ and $(S_2,\phi_2)$ are commensurable, 
the mapping tori associated to these automorphisms are 3-manifolds that 
admit a common finite-sheeted cover, 
that is, are commensurable as 3-manifolds \cite{Walsh}. 
The especial interest of the pseudo-Anosov 
case is that the interior of the mapping torus of $(S,\phi)$ is a complete 
hyperbolic 3-manifold that fibers over $S^1$, 
by work of Thurston \cite{Otal}

If we have a covering $p\colon \wt S \to S$, 
the set of elements of $\MCG(\wt S)$ that cover an element of $\MCG(S)$ 
under $p$ is a subgroup $H_p$, 
and the covering relation is a homomorphism 
$p_\#\colon H_p \to \MCG(S)$. 
This is because the 
covering relation $p_\#$ distributes over composition, 
in the sense that if 
$p\colon \wt \phi \to \phi$ and $p\colon \wt \psi \to \psi$, then
\[p (\wt \phi \wt \psi) = \phi p \wt \psi = \phi \psi p,\]
i.e., $\phi \psi = p_\#(\wt \phi \wt \psi)$.
Thus if we have two covering maps $p_j \colon \wt S \to S_j$, 
the set of elements of $\MCG(\wt S)$ that cover both 
through $p_1$ elements of $S_1$ and
through $p_2$ elements of $S_2$ is a subgroup $H$, 
and elements of $(p_1)_\# H$ are each commensurable with some element of 
$(p_2)_\#(H)$.

Associated with each $(S,\phi)$, where $\phi$ is pseudo-Anosov, 
are a number of 
invariants that are useful in detecting commensurability, as given in \cite{CSW}:

\begin{enumerate}
\item whether or not $\d S = \0$; 
\item the commensurability class of the mapping torus, 
\item the commensurability class in $\R$ of $\log(\l)$, where 
$\l$ is the expansion factor of the
pseudo-Anosov homeomorphism $f \in \phi$; 
\item the set of orders of the singular points of the invariant foliations of $f$.
\end{enumerate}

The set of orders of singular points in the invariant foliations of $f$ can 
be more helpfully viewed as an infinite vector.

\begin{definition}
Define $\delta_n(f)$ to be the number of $n$-prong singularities 
in the invariant foliation associated to a pseudo-Anosov $f$, 
and write $\delta(f) = (\delta_n(f))_{n \in \N^+ \less \{2\}}$ 
for the infinite-length vector describing the \emph{singularity data} for that foliation. 
(The case $n=2$ is excluded because a 2-prong is not a singularity.)
\end{definition}
In \cite{CSW} it is shown that the rational commensurability class of 
$\delta(f)$ is a commensurability invariant of the mapping class $\phi$, 
so that a pair of pseudo-Anosovs $\phi$ and $\psi$ can be commensurable 
only if there is some rational number $q \in \Q^+$ such that 
$q\delta(\phi) = \delta(\psi)$; that is, 
$q\delta_n(\phi) = \delta_n(\psi)$ for each $n$. 

\section{Thurston's two-multitwist construction of pseudo-Anosovs}\label{multitwist}

Let $a$ and $b$ be multicurves on a surface $S$, 
and write $T_a,T_b$ for the mapping classes 
of the Dehn multitwists around $a$ and $b$. 
Write $G(a,b)$ for the subgroup $\ang{T_a,T_b}$ 
of the mapping class group $\MCG(S)$.
We now, following \cite{Thur}, 
create a flat structure on $S$ depending on $a$ and $b$. 
To do this, we first create the {\em dual cell decomposition} $\Sigma(a,b)$ 
of $S$ into rectangular cells, 
dual to the decomposition of $S$ determined by $a \cup b$;
we then assign lengths to each rectangle in such a way that $T_a$ and $T_b$ 
act by affine transformations.

To get $\Sigma(a,b)$, draw a small rectangle in the surface 
at each intersection of an 
$a$ curve and a $b$ curve, with sides transverse to the curves. 
Expand these rectangles until the surface is covered.
After expansion, a rectangle at an intersection point $p$ of $a_j$ and $b_k$ 
will share a side with just the rectangles drawn around intersection points 
on $a_j$ and on $b_k$ that are adjacent to $p$. 
Thus rectangles are ultimately glued if they lie along the same curve 
and there are no intervening rectangles between them on the same curve.
This is called the \emph{dual cell decomposition} $\Sigma = \Sigma(a,b)$ to that 
determined by $a \cup b$, and if the cell decomposition of $S$ determined 
by $a \cup b$ has $V$ vertices, $E$ edges, and $F$ faces, 
the dual decomposition has $F$ vertices, $E$ edges, and $V$ faces. 
For example, 
if $a \cup b$ does not separate $S$, 
the cell decomposition has one face, 
so the dual cell decomposition will have one vertex.
In general, one can show that 
if a face in the cell decomposition determined by $a \cup b$ 
is bounded by $2n$ edges, the corresponding vertex in the dual cell decomposition 
will be an $n$-prong if the squares are foliated by lines parallel to the $a$ curves.
The natural product foliations of the Euclidean rectangles induce 
a pair of transverse singular foliations on $S$; 
the singularities are those vertices of the dual cell decomposition 
that are the result of identifying $2n \neq 4$ rectangle corners. 
Call $\delta(a,b)$ the infinite vector whose $n$th entry is the number of 
$n$-prongs of either foliation
The entries of $\delta(a,b)$ are intimately related with the genus $g$ of $S$, 
by the Hopf index theorem: 
if $S$ is closed, we have $4(g-1) = \sum_{n} (n-2) \delta_n$.
Thus if $a \cup b$ does not separate $S$, we have 
$4g - 4 = n - 2$, so $\delta = e_{4g-2}$; 
that is, the single singularity of the foliation 
is a $(4g-2)$-prong. 

\bigskip
\begin{figure}
{
\includegraphics{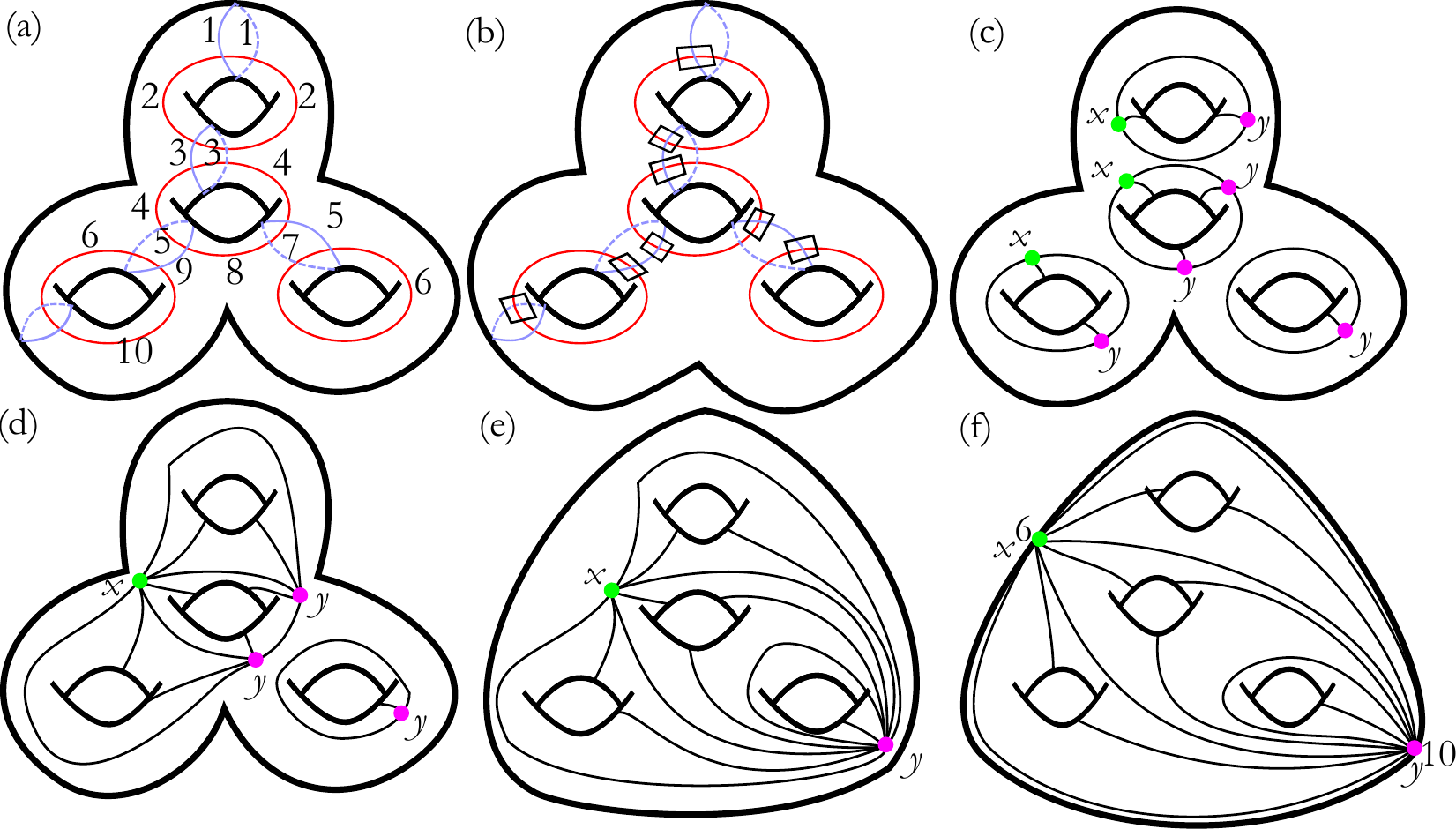}
}
\caption{Drawing a dual cell decomposition}
\end{figure}

To illustrate the formation of the dual cell decomposition, we do an example. 
In part (a) of the figure on the above, 
we have drawn a multicurve configuration $a,b$ filling a surface $S$. 
We see that $S \less (a \cup b)$ has two components. If we count the number 
of edges bounding the left region, 
we see there are 12 (six in front, and another six on the side we can't see) 
and 20 edges bounding the right one. 
In (b), we draw squares, 
which we will grow to be faces of the dual decomposition $\Sigma(a,b)$, 
centered at the vertices $a \cap b$. 
Now there are six squares along the edges bounding the left region of 
the original cell decomposition. 
Altogether, they have 24 corners, 12 of which lie in the left region, and 12 in the right. 
Since the dual decomposition is to have one vertex corresponding to the left 
region, when we expand the squares to fill the surface, the left vertex 
will have a total angle of $12(\pi/2) = 6\pi$. 
Since each separatrix has an angle of $\pi$, 
then, the left vertex of $\Sigma(a,b)$ will be a 6-prong.
Similarly, the right region contains $6\cdot 2 + 2 \cdot 4 = 20$ square corners, 
so the right vertex of $\Sigma(a,b)$ will be a 10-prong.
In (c), we move the centers of the squares in such a way that 
half of each square lies along the front of the surface, and half along the back, 
which we can't see. Since the picture along the back is symmetrical, 
we can focus on the front, implicitly doing the same thing in the back.
We expand the squares along the red curves until adjacent squares along 
the red curves have had their sides identified. 
We color the square corners that are to be identified to the left vertex 
in $\Sigma(a,b)$ light green, and label them $x$, 
and color those to be identified with the right one a darker fuchsia, 
and label them $y$.
In (d), we make two side identifications of squares. 
There is one green corner left in the front of the picture, 
and three fuchsia ones, which are yet to be identified. 
In (e), we identify these three fuchsia endpoints. 
This involves
wrapping the edge with two fuchsia endpoints around the 
circle on the right with one fuchsia endpoint.
Once we've done that, 
the picture has one fuchsia point in front and one in back, and 
one green point in front and one in back.
The only sides of squares left unidentified are two outermost ones in the front 
and their mirror images in back, all parallel to the thick black border of the picture. 
In (f), we drag these edges and vertices to the black border and identify them. 
This finally is the dual cell decomposition.

In order to assign lengths
to $a = \{a_1,\ldots,a_m\}$ and $b = \{b_1,\ldots, b_n\}$, 
associate a bipartite graph $\Gamma(a,b)$ as follows. 
For each curve $a_j$ and $b_k$ create a vertex, 
and create one edge between $a_j$ and $b_k$ for each intersection. 
Associated with this graph is $N = N(\Gamma(a,b))$,
an $|a| \x |b|$ incidence matrix whose $(j,k)$ entry is 
$(N)_{j,k} = i(a_j,b_k)$, the geometric intersection number of $a_j$ and $b_k$. 
(or the number of edges between $a_j$ and $b_k$ in $\Gamma(a,b)$.) 

Assume the graph $\Gamma(a,b)$ is connected; 
then the square matrix $NN^\top$ has nonnegative integer entries, 
and some power of $NN^\top$ has 
strictly positive entries. It therefore has a \emph{Perron--Frobenius eigenvalue} 
\cite{Gant}: 
there is a unique positive real eigenvalue $\mu = \mu(a,b)$ of multiplicity one and
such that $\mu > |\mu'|$ for all other eigenvalues $\mu' $of $NN^\top$.
The Perron--Frobenius eigenvalue has a 
\emph{Perron--Frobenius eigenvector} $v$ all of whose entries are 
positive. Let $v' = \mu^{-1/2} N^\top v$. 
Then $Nv' = \mu^{-1/2} NN^\top v = \mu^{1/2} v$, 
so $v= \mu^{-1/2} N v'$, 
and $v'$ is the Perron--Frobenius eigenvector for $N^\top N$, since
\[N^\top N v'
= \mu^{-1/2} N^\top N N^\top v 
= \mu^{-1/2} N^\top (\mu v)
= \mu \mu^{-1/2} N^\top v
= \mu v'.\] 
We complete the flat structure $X = X(a,b)$ by 
giving each rectangle corresponding to 
an intersection of $a_j$, considered as running horizontally and $b_k$, 
considered as running vertically, 
the metric structure of the Euclidean rectangle $[0,v'_k] \x [0,v_j]$. 

To see that the multitwists $T_a$ and $T_b$ act in an affine way on this structure, 
consider doing a Dehn twist around the curve $a_j$.
The squares containing arcs of $a_j$ are the $\sum_k i(a_j,b_k)$ squares 
induced by intersections with curves $b_k$. 
These squares are isometric to $[0,v'_k] \x [0,v_j]$, 
so their union is a cylindrical neighborhood of $a_j$ with height $v_j$ 
and circumference 
\[\sum_k i(a_j,b_k) v'_k  
= (N v')_j 
= (\mu^{-1/2} NN^\top v)_j 
= \mu^{-1/2} \mu v_j = \mu^{1/2} v_j.\]
Putting coordinates in $[0, \mu^{1/2} v_j] \x [0,v_j]$ on this cylinder,
doing a right-handed Dehn twist on this cylinder takes 
$(x,y) \mapsto (x + \frac y{v_j} (\mu^{1/2} v_j), y ) = (x + y \mu^{1/2},y)$. 
Since this is the same on every cylinder, $T_a$ has derivative given everywhere 
by \[\mat{1 & \sqrt \mu \\ 0  & 1}.\]
Similarly, taking all the squares meeting $b_k$ 
gives a cylinder neighborhood with height $v'_k$ 
and circumference 
\[\sum_j v_j \cdot i(a_j,b_k) 
= (v^\top N)_k
= (N^\top v)_k
= (\mu^{-1/2} N^\top N v')_k
= \mu^{1/2} v'_k,
\]
and a Dehn multitwist around $b$ (right-handed!) 
gives an affine map with derivative $\mat{1&0\\-\sqrt\mu&1}$.

By the chain rule, then, an element $f$ of $G(a,b)$ always has a derivative $Df$. 
Since the matrices have determinant $1$, the eigenvalues of $Df$ 
multiply to $1$. In fact, the eigenvalues of $Df$ are
given by $x^2 - \tr(Df) x + 1$, so the eigenvalues are complex conjugates, 
both $\pm 1$, or real $>1$ and inverse depending on whether the absolute value of 
the trace is less than, equal to, or greater than $2$, respectively.
If $Df$ has real eigenvalues, 
then invariant foliations for $f$ are given by the eigenvectors for $Df$, 
and leaves of these foliations are expanded and contracted by their 
corresponding eigenvalues $\l^{\pm 1}$. 
Thus if $Df$ has real eigenvalues, $f$ is a pseudo-Anosov homeomorphism. 
The singularities of the natural foliation on $X(a,b)$ give rise to the singularities
of the invariant foliations for $f \in G(a,b)$.
This is part of the following theorem.

\begin{theorem}[Thurston \cite{Thur}]\label{rep}
The derivative of the action of the multi-twists on the flat structure gives rise to 
a representation $\rho\colon G(a,b) \to \PSL(2,\R)$, 
given by
\[
T_a \mapsto \mat{1&\sqrt\mu\\0&1}, \quad
T_b \mapsto \mat{1&0\\-\sqrt\mu&1}.
\]
The kernel of $\rho$ is finite and 
there is some $k$-fold covering group $G_k \to \PSL(2,\R)$ 
such that $\rho$ lifts to a faithful representation in $G_k$; 
if $a \cup b$ fills $S$, then $k$ is finite.
The image of $\rho$ is a discrete subgroup of $\PSL(2,\R)$, 
and is free just if $\sqrt \mu \geq 2$.
For $g \in G(a,b)$, 
the image $\rho(g)$ 
is elliptic (or the identity), parabolic, or hyperbolic just if $g$ 
is (respectively) finite-order, reducible, or pseudo-Anosov. 
If $\rho(g)$ is reducible, some power of $g$ is a multi-twist, 
and if $\rho(g)$ is hyperbolic, its larger eigenvalue 
is the expansion factor $\l(g)$ of $g$.
\end{theorem}

If $a \cup b$ fills, 
then Thurston showed ``most of'' the classes in the two-generator subgroup 
$G(a,b) := \ang{T_a,T_b}$ of $\MCG(S)$ are pseudo-Anosov
--- all of them except $\ang{T_a}$, $\ang{T_b}$, 
and possibly $T_a^{\pm 1}T_b^{\pm 1}$.

\begin{lemma}
If $g \in G(a,b)$ is pseudo-Anosov, 
the expansion factor $\l(g)$ 
is an algebraic integer quadratic over $\mathbb{Z}[\mu(a,b)]$.
\end{lemma}
\begin{proof}
Note that $\rho(T_a) = \mat{1&\sqrt \mu\\0 & 1}$ 
and $\rho(T_b) = \mat{1 & 0 \\ -\sqrt \mu & 1}$ 
have diagonal entries (here $1$) in $\mathbb{Z}[\mu]$ 
and anti-diagonal entries in $\sqrt \mu \cdot \mathbb{Z}[\mu]$.
The product of two such matrices again has this form:
\[\mat{p(\mu) & q(\mu)\sqrt{\mu}  \\ r(\mu) \sqrt{\mu} & s(\mu)}
\mat{p'(\mu) &  q'(\mu)\sqrt{\mu} \\ r'(\mu) \sqrt{\mu} & s'(\mu)}
 =
\mat{(pp')(\mu) + \mu (qr')(\mu) & \sqrt{\mu} (pq' + qs')(\mu) \\ 
\sqrt{\mu}(rp'+sr')(\mu) & \mu (rq')(\mu) +  (ss')(\mu)}
.\]
In particular, the trace of such a matrix is in $\mathbb{Z}[\mu]$, 
so since the determinant is $1$, 
the larger eigenvalue $\l(g)$ satisfies 
the equation $y^2 - p(\mu)y + 1 = 0$ 
for some $p(x) \in \mathbb{Z}[x]$.
One has $\l = \frac 1 2\paren{p(\mu) + \sqrt{p(\mu)^2 - 4}}$.
\end{proof}

If two pseudo-Anosovs $g \in G(a,b)$ and $h \in G(a',b')$ are commensurable, 
then there are $n,m \in \mathbb{Z}\less\{0\}$ 
such that $\l(g)^n = \l(h)^m$, 
and these powers are in $\Q(\l(g)) \inter \Q(\l(h))$.
Writing $\mu = \mu(a,b)$ and $\mu' = \mu(a',b')$, 
then there are polynomials $p,q \in \mathbb{Z}[x]$ 
and $r,s,t,u \in \Q[x]$ 
such that \[r(p(\mu)) +s(p(\mu)) \sqrt{p(\mu)^2 - 4} = 
t(q(\mu')) + u(q(\mu'))\sqrt{q(\mu')^2 - 4}.\]
This would seem to put rather serious constraints on $\l(g),\l(h)$. 
For example, if $g,h \in G(a,b)$, 
then we have $\Q(\mu,\l(g))$ and $\Q(\mu,\l(h))$ quadratic over $\Q(\mu)$, 
so if $g,h$ are commensurable, then $\l(g)$ and $\l(h)$ have some common power, 
and if it isn't in $\Q(\mu)$, then $\Q(\mu,\l(g)) = \Q(\mu,\l(h))$. 
But a power of $\l$ isn't in $\Q(\mu)$ unless $\sqrt{p(\mu)^2 - 4} \in \Q(\mu)$, 
in which case $\l$ itself is in $\Q(\mu)$. 

We might compare, for example, the discriminants of $\Q(\mu,\l(g))$ 
and $\Q(\mu,\l(h))$ over $\Q(\mu)$ or $\Q$ with a number theory program.

\section{The invariant $J$}\label{J}

Associated to a flat surface $X$ (see \cite{Zor} for definitions), 
Kenyon and Smillie \cite{KS} 
associate an invariant 
$J$ in the rational vector space $\R^2 \wedge_\Q \R^2$. 
To define it, one starts with a decomposition of the flat structure into planar polygons.
For each planar polygon $P$, with vertices $v_1,\ldots,v_n$, 
define \[J(P) = v_n \wedge v_1 + \sum_{j=1}^{n-1} v_j \wedge v_{j+1}.\]
Then if $X$ can be decomposed as a union of planar polygons 
$X = \Union_{j=1}^n P_j$, 
glued along edges by translations,
define $J(X) = \sum_j J(P_j)$. 

One can show that $J(X)$ is independent of the decomposition and 
that $J(P)$, $P$ a polygon is independent of translations.
$J$ does vary under rotations, though. Suppose we
rotate a polygon by an angle of $\theta$ around the origin.
Making the canonical identification $\C \iso \R^2$, 
the new vertices are given by $e^{i\theta}v_j$,
and so the new $J(P)$ is 
$e^{i\theta} v_n \wedge e^{i\theta} v_1 
+ \sum_{j=1}^{n-1} e^{i\theta}  v_j \wedge e^{i\theta} v_{j+1}$. 
In particular, rotation by $\pi$ leaves $J$ unchanged.

A useful formula for us is the following. 
Define the {\em edge vectors} by 
$e_1 = v_1 - v_n$ and $e_j = v_j - v_{j-1}$ for $j > 1$. 
Then for a rectangle $R$ with vertices $v_1,v_2,v_3,v_4$ arranged to be horizontal 
and vertical in the plane, we have $J(R) = 2e_1 \wedge e_2$.

Given a covering of surfaces $p\colon \wt S \to S$, we say it is a \emph{covering 
of flat surfaces} if we can put flat structures $\wt X$ and $X$ on $\wt S$ and $S$ 
given by quadratic differentials $\wt q$ and $q$ such that $p^* q = \wt q$ 

\begin{lemma}
Given an $n$-fold covering $p \colon \wt X \to X$ of flat surfaces, 
$J(\wt X) = n J(X)$.
\end{lemma}
\begin{proof}
Let a convex polygonal decomposition $X = \Union_j P_j$ be given.
Since each polygon $P_j$ is simply-connected, 
the inclusions $P_j \inc X$ lift to $\wt X$, 
so $\wt X$ is tiled by lifts of these polygons. 
There are $n$ copies of each $P_j$ in this decomposition of $\wt X$, 
so $J(\wt X) = n J(X)$.
\end{proof}

A pseudo-Anosov $(S,\phi)$
determines a pair of transverse measured singular foliations $\mc F^u$ and 
$\mc F^s$ on $S$, unique up to a multiplicative constant for the measures. 
These measured foliations 
in turn determine a quadratic differential that evaluates to positive 
real numbers at tangent vectors to the unstable foliation and to negative reals 
at tangent vectors to the stable foliation. This gives a flat structure preserved 
by the pseudo-Anosov, uniquely defined up to scale. 
That is, consider the map $\psi_{r,s}\colon \R^2 \to \R^2$ given by 
$(x,y) \mapsto (rx,sy)$ for $r, s \in \R^\x$, 
which induces a map $\Psi_{r,s}\colon 
\R^2 \wedge_\Q \R^2 \to \R^2 \wedge_\Q \R^2$ 
by $v \wedge w \mapsto \psi_{r,s}(v) \wedge \psi_{r,s}(w)$; 
$J$ is determined up to the map $\Psi_{r,s}$.

\begin{theorem}
For two commensurable pseudo-Anosovs $\phi_1$ and $\phi_2$, 
with associated flat surfaces $X_1$ and $X_2$, 
there are $r,s \in \R^\x$ with such that 
$\Psi_{r,s} (J(X_1)) = J(X_2)$.
\end{theorem}

\begin{proof}
If $\phi_1$ and $\phi_2$ are commensurable, their invariant foliations lift to 
the same foliations $\wt {\mc F}^s$, $\wt{\mc F}^u$ 
under the surface coverings $p_j\colon \wt S \to X_j$. 
Since the common lift $\wt \phi$ must alter the transverse measures according to 
$\wt \phi \wt {\mc F}^s = \frac 1 \l  \wt {\mc F}^s$ and
$\wt \phi \wt {\mc F}^u = \l  \wt {\mc F}^u$, the transverse measures are uniquely 
determined up to one scaling factor each. 
Thus we can alter the measures on $\wt {\mc F}^s$ and $\wt {\mc F}^u$ 
induced by lifting the invariant foliations of $\phi_1$ 
to be equal to those induced by $\phi_2$.
This induces $J(X_1) \mapsto \Psi_{r',s'} (J(X_2))$ for some $r',s' \in \R^+$.

It may be that one of the covering maps $p_j$ is orientation-reversing. In case this 
happens, we alter the flat structure of one of our surfaces, say $X_1$ 
by the map $x \mapsto -x$. 
This induces the transformation $J(X_1) \mapsto - J(X_1)$.

Now the two flat surfaces $X_1$ and $X_2$ have a common covering flat surface, 
then there is $q \in \Q^+$ such that $q J(X_1) = J(X_2)$.
If we define the stable foliation to run 
``north-south'' and the unstable to run ``east-west,'' 
then the frame is determined up to a rotation by $\pi$. 

Putting this all together, we have $\pm q \Psi_{r,s} (J(X_1)) = J(X_2)$.
\end{proof}

One free variable can be removed 
by requiring the total area of the surface to be $1$,
but there is more flexibility in this construction than we would like. 

The natural invariant foliations for different pseudo-Anosovs 
in a group $G(a,b)$ generally point in different directions, 
so the associated flat structures for different pseudo-Anosovs differ by 
$v\wedge w \mapsto e^{i\theta} v \wedge e^{i\theta} w$. 
However, once we fix a scaling, 
this is the only difference between different flat structures for
pseudo-Anosovs in a group $G(a,b)$.

\section{Some two-multitwist groups and their associated invariants}
\label{invariants}

Leininger \cite{Lein} found precise conditions 
on the intersection graph $\Gamma(a,b)$ of the multicurves $a$ and $b$ 
for the subgroup $G(a,b)$ to be free:
it is free just if the graph has some component that is not among the graphs
$\mc{A}_j, \mc{D}_j,\mc E_6, \mc E_7, \mc E_8,  j \in \N$. 
The Teichm\"{u}ller curves for which the 
associated stabilizers contain with finite index 
a group generated by two positive multi-twists 
are these and others corresponding to graphs $\mc P_{2j}, \mc Q_j, 
\mc R_7, \mc R_8, \mc R_9$, $j \in \N$. 
These graphs are Dynkin diagrams with simple edges,
pictures of which will appear in Figures 2--11.

In this section, we describe the invariants $\delta (a,b))$, $\mu(a,b)$, and $J(X(a,b))$ 
for certain multicurve configurations $a,b$.
For pairs of multicurve configurations $a,b$ and $a',b'$, 
if $\delta(G(a,b))$ and $\delta(G(a',b'))$ are not rationally commensurable, 
then no pseudo-Anosov element of $G(a,b)$ 
is commensurable with any in $G(a',b')$. 
The same is true if the invariants $J$ differ other than by the action of 
$S^1 \x \R^+$.

In all of our pictures, $a$ will be the red curves and $b$ the light blue ones.
For a bipartite graph $\Gamma$ we show 
\begin{enumerate} 
\item a multicurve configuration $a,b$ on a surface $S$ 
such that $\Gamma(a,b) = \Gamma$;
\item $\Gamma$ itself;
\item the dual cell decomposition $\Sigma(a,b)$ corresponding to invariant foliations 
for pseudo-Anosovs in $G(a,b)$, with singularity orders written at vertices;
\item the flat structure $X(a,b)$ minus length information---this is basically 
the picture of $\Sigma(a,b)$ cut along a few edges and straightened out.
Note that this picture is only an approximation; the cells are \emph{not really} 
squares, but are just notated that way for uniformity of presentation.
\end{enumerate}
We sometimes omit the picture of $\Sigma(a,b)$ on the surface 
in preference for the square-tile picture.
The edge labelling in the former is preserved in the latter when both are presented.
In the picture of $\Sigma(a,b)$, 
only one side of the surface is shown, for clarity, which amounts to 
an assumption the surface is not transparent.
The other side looks the same, 
except for edges that would otherwise be 
along the dark black boundary,
which I have pushed in into the visible side.
The labels in parentheses are for ``invisible'' edges 
that lie wholly on the other side of $S$.
\bigskip

We remark that the pictures we have drawn are essentially unique 
in the following sense.

\begin{lemma}[Leininger \cite{Lein}]\label{unique}
Suppose $a \cup b$ fills $S$ and $a' \cup b'$ fills $S'$, 
and their incidence graphs $\Gamma(a,b)$ are the same. 
If $\Gamma(a,b)$ is a tree with all but possibly one vertex of valence $\leq 2$ 
and the remaining vertex of valence $\leq 3$, 
then there is a homeomorphism $S \to S'$ taking $a \cup b \to a' \cup b'$, 
up to adding marked points.
\end{lemma}

\newpage

Our first item is the intersection graph $\mathcal A_{n}$. 
Shown below are cases $n=2,4,6$. 
The groups $G(\mc A_n)$ are \emph{not} free, according to Leininger's result. 
The graph $\mc A_n$ determines a filling curve configuration uniquely 
up to conjugacy and adding punctures, by Lemma 3, 
so our picture is essentially unique.
Since $a \cup b$ does not separate
the surface if $n$ is even, 
we get that the singularity data $\delta(\mc A_{2n})$ is $e_{4n-2}$, 
a single $(4n-2)$-prong singularity. 
(If $n=1$, $\delta = e_2 = 0$ has no singularities.)
Since these vectors are rationally incommensurable for different $n$, 
if $\phi \in G(\mc A_{2n})$ and $\psi \in G(\mc A_{2m})$ are commensurable 
pseudo-Anosovs, we must have $m = n$.
\bigskip

\begin{figure}[h]
{
\includegraphics{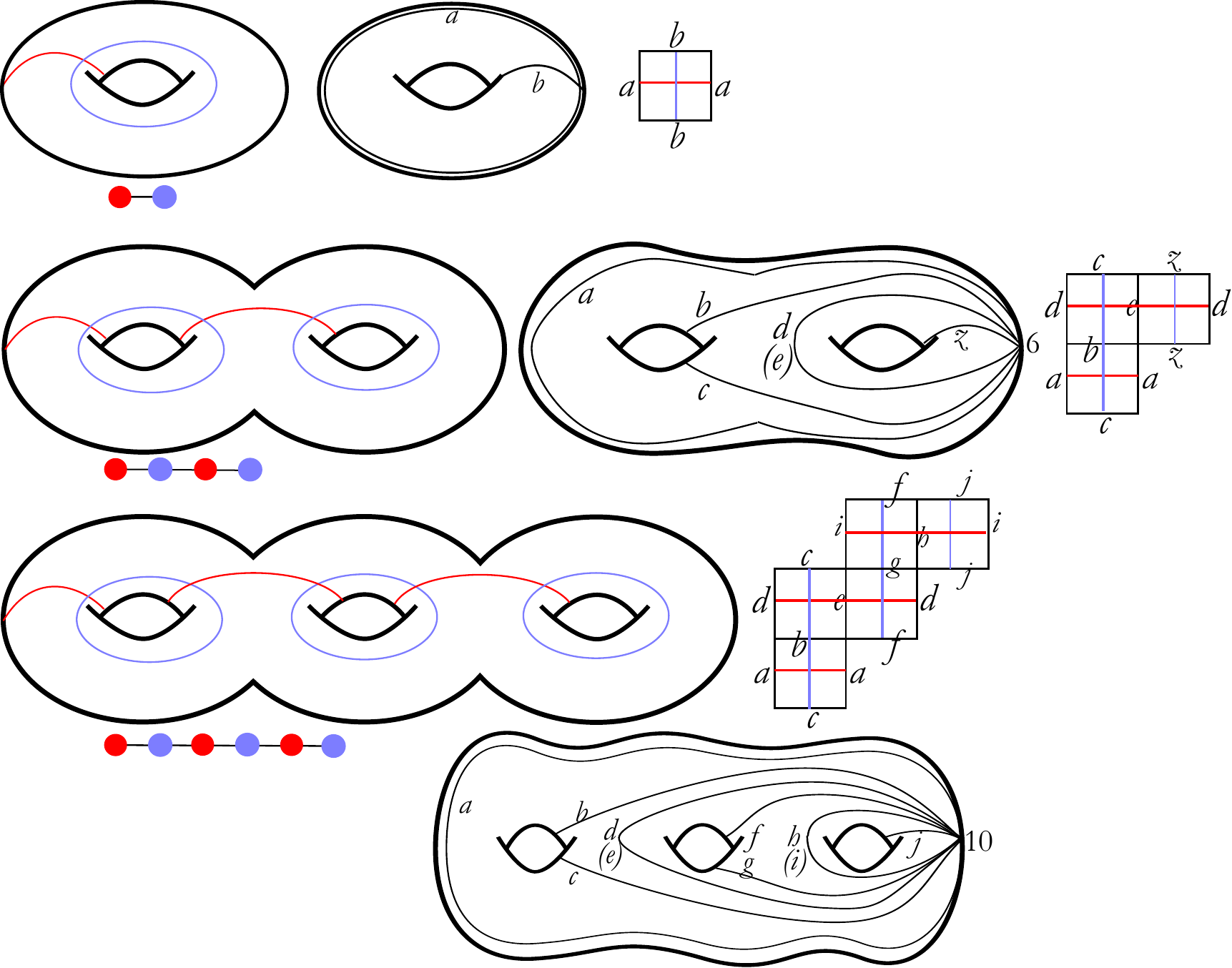}
}
\caption{$\mc A_{2n}$}
\end{figure}
The incidence matrix $N(\mc A_2) = (1)$, so $\mu = 1$.
The incidence matrix $N(\mc A_4) = \mat{1 & 0 \\ 1 & 1}$, 
so $NN^\top = \mat{1&1\\1&2}$ and $\mu = \frac 1 2(3 + \sqrt 5)$.
An eigenvalue $\mat{x\\y}$ then satisfies $x + y = \mu x$, or 
$y = (\mu -1)x = \frac {1 + \sqrt 5}{2} x$, where $\mu - 1 = \sqrt \mu = \g$ 
is the golden ratio. 
Arbitrarily setting $x = 1$, we get a Perron--Frobenius eigenvector 
$v = \mat{1\\ \g }$. 
The associated eigenvector for $N^\top$ is 
$\mu^{-1/2} N^\top v = \g^{-1} \mat{\g^2 \\ \g } = \mat{\g \\ 1} = v$ again.
So the flat strucure associated to $\mc A_4$ 
consists of three rectangles of proportions $\g \x 1$, $\g \x \g$, and $1 \x \g$. 
For this reason, this table is called the {\em golden table}; see  \cite{McM}.
The $J$-invariants for these rectangles are respectively $2(\g,0)\wedge(0,1)$,
$2(\g,0) \wedge(0,\g)$, and $2(1,0)\wedge(0,\g)$, 
and so the $J$-invariant for this L-shaped table is 

\[2[(\g,0)\wedge(0,1) + (\g,0) \wedge(0,\g) + (1,0)\wedge(0,\g)]. \] 

In general the incidence matrix $N(\mc A_{2n})$ has $1$ on the diagonal 
and subdiagonal, and $0$ elsewhere: 
$\mat{1 & 0 & 0 & \cdots & 0&0\\
            1 & 1 & 0 & \cdots & 0&0\\
            0 & 1 & 1 & \cdots & 0&0\\
            \vdots&\vdots&\vdots&\ddots&\vdots&\vdots\\
               0 & 0 & 0 & \cdots &1& 0\\
            0 & 0 & 0 & \cdots &1& 1}$.
Thus $NN^\top =
 \mat{1 & 1 & 0 & \cdots &0 & 0 \\
	1 & 2 & 1 & \cdots &0 & 0 \\
	0 & 1 & 2 & \cdots &0 & 0\\ 
          \vdots&\vdots&\vdots&\ddots&\vdots&\vdots\\
             0 & 0 & 0 & \cdots & 2&	1\\
             0 & 0 & 0 & \cdots & 1&	2}$
has $(1, 2, \ldots, 2)$ along the diagonal, 
$1$ along the super- and subdiagonals, and $0$ elsewhere.

Note that up to reflections, the pattern of $1$s in $N$ is the same as 
the pattern of crossings in the flat structure.

\newpage
Shown below are configurations for the intersection graph $\mc A_{2n+1}$, 
$n=1,2,3$. 
The singularity data are $2e_{2n}$; that is, there are two $2n$-prong singularities 
and no others. (If $n = 1$, $\delta = 2e_2 = 0$.)
Again, the different $G(\mc A_{2n+1})$ have no commensurable pseudo-Anosovs. 
The singularity data doesn't rule out commensurable pseudo-Anosovs 
in $G(\mc A_{2m})$ and $G(\mc A_{2n+1})$ 
if $4m-2 = 2n$, so $n = 2m-1$. 
In this case the surface carrying the configuration corresponding to $\mc A_{2n+1}$ 
double covers that for $\mc A_{2m}$.
Commensurability between different elements of these groups must be ruled out in 
other ways.

\begin{figure}[h]
{
\includegraphics{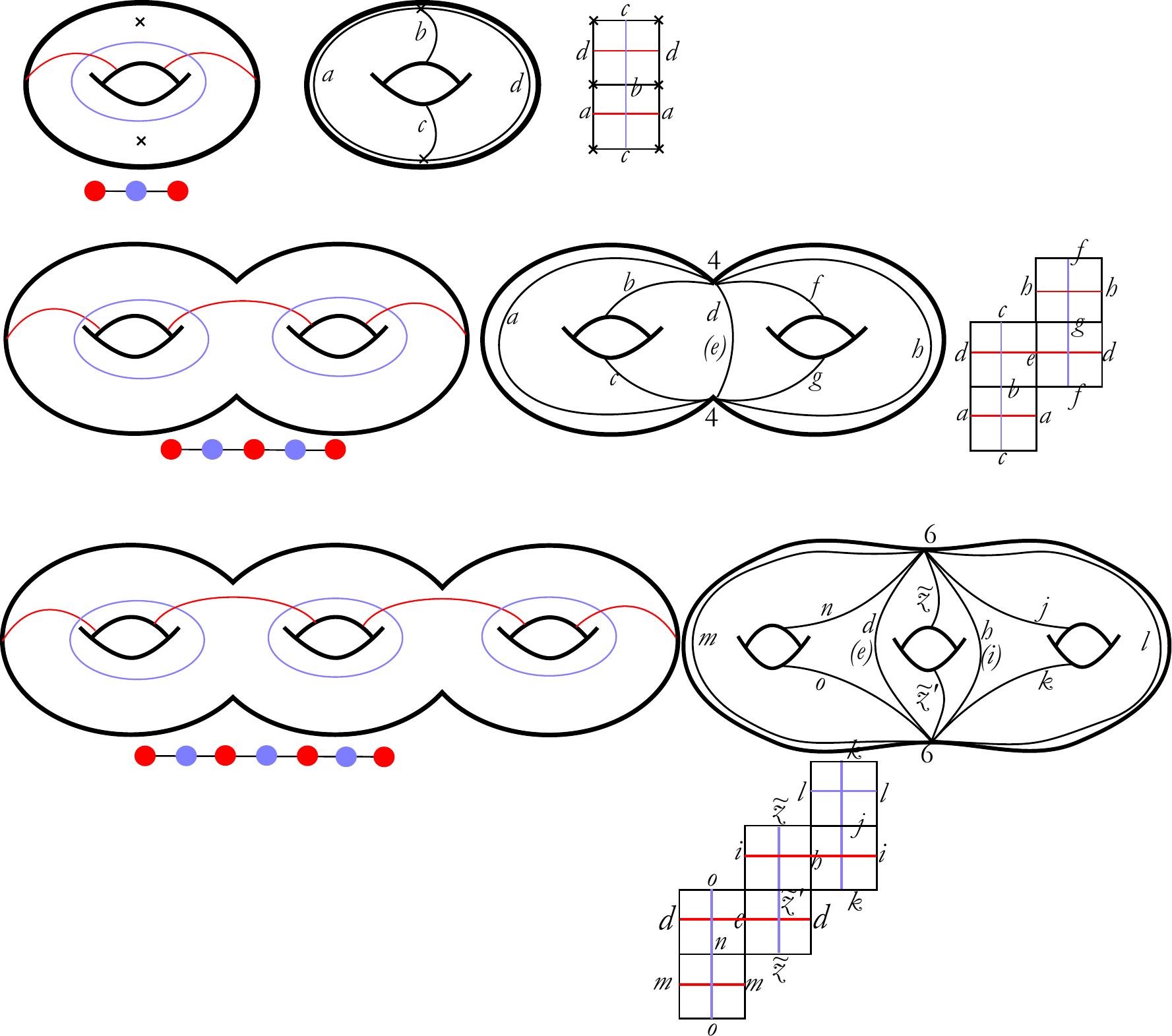}
}
\caption{$\mc A_{2n+1}$}
\end{figure}
The incidence matrix $N(\mc A_3) = \mat{1 \\1}$, 
so that $N^\top N = (2)$ and $\mu = 2$.
An eigenvector is $(1)$, 
and $2^{-1/2} N^\top (1) = \mat{2^{-1/2}\\2^{-1/2}}$, 
so the flat structure $X(a,b)$ is two rectangles of dimensions $2^{-1/2} \x 1$. 
The invariant $J$ is then $4 (2^{-1/2},0) \wedge (0,1)$.

The incidence matrix $N(\mc A_5) = 
\mat{1 & 0\\
	1 & 1\\
	0 & 1}$, 
so that $N^\top N = \mat{2&1\\1&2}$ and $\mu = 3$. 
An eigenvector is $v' = \mat{1\\1}$, 
and $3^{-1/2}N v = 3^{-1/2} \mat{1\\2\\1}$.
The invariant $J$ is then 
$4 (1,0)\wedge (0,1/\sqrt 3) + 4 (1,0) \wedge (0,2/\sqrt{3})$. 

In general, $N(\mc A_{2n+1})$ is a $(n+1) \x n$ matrix with $1$s  
on the diagonal and subdiagonal and $0$ elsewhere, 
so that $N^\top N$ is an $n \x n$ matrix with $2$ along the diagonal, 
$1$ on the super- and subdiagonals, and $0$ elsewhere:
$ \mat{2 & 1  & \cdots & 0 \\
	1 & 2 &  \cdots & 0 \\
          \vdots&\vdots&\ddots&\vdots\\
            0 & 0 & \cdots & 2}$. 
In particular $\mu(\mc A_7) = 2 + \sqrt 2$, 
so $N^\top N$ has eigenvector $v' = \mat{1 \\ \sqrt 2 \\ 1}$, 
and $NN^\top$ has eigenvector 
$v = (2+\sqrt 2)^{-1/2} \mat{1\\1+\sqrt 2\\ 1+ \sqrt 2 \\1}$.
The corresponding $J$-invariant is 
$ 4(1,0)\wedge(0,(2+\sqrt 2)^{-1/2}) 
+ 4(1,0)\wedge(0,\frac{1+\sqrt 2}{\sqrt{2+\sqrt {2}}})
+ 4 (\sqrt 2,0) \wedge (0,\frac{1+\sqrt 2}{\sqrt{2+\sqrt 2}})$. 
Comparing this with the L-shaped table of $\mc A_4$, 
we see that the covering of $\mc A_4$ by $\mc A_7$ is not a covering of flat 
surfaces, so pseudo-Anosovs in $G(\mc A_4)$ do not lift to elements of $\mc A_7$.

\newpage
Shown below are configurations for the intersection graph $\mc D_{2n}$, 
$n=2,3,4$. 
The singularity data are $\delta(\mc A_{2n+1}) = 2e_{2n} = \delta(\mc D_{2n+2})$,
so they don't rule out commensurability between pseudo-Anosovs in these groups.
Again, the $G(\mc D_{2n})$ for different $n$ have no 
commensurable pseudo-Anosovs, and they also have no pseudo-Anosovs 
commensurable with elements of the $G(\mc A_{2n})$.

\bigskip
\begin{figure}[h]
\includegraphics{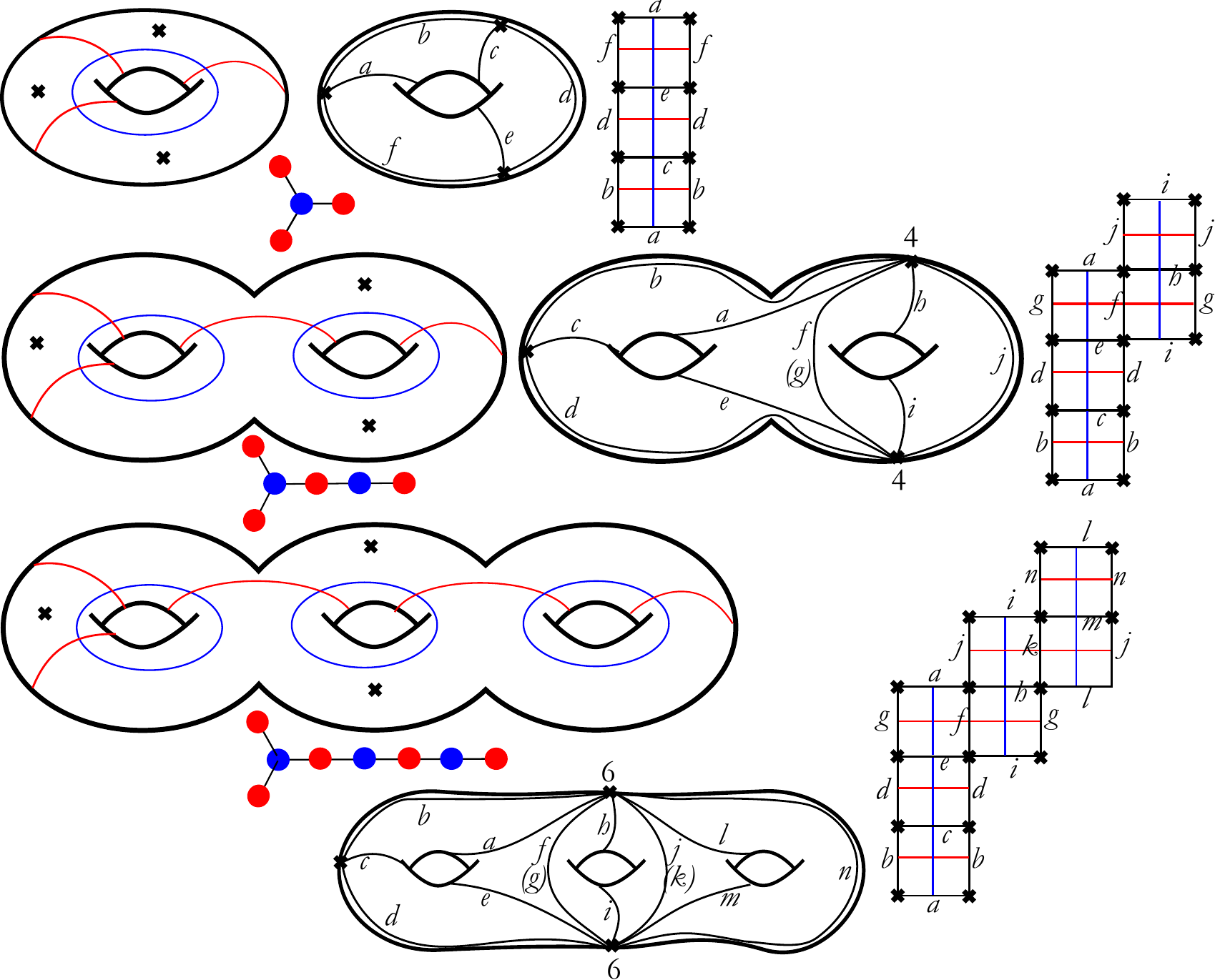}
\caption{
$\mc D_{2n}$
}
\end{figure}
$N(\mc D_4) = \mat{1\\1\\1}$, so $N^\top N = (3)$ and $\mu = 3$, 
and an eigenvector for $N^\top N$ is $v' = (1)$. 
Then $v = 3^{-1/2}Nv' = \mat{3^{-1/2} \\ 3^{-1/2} \\ 3^{-1/2}}$, 
so $J = 6(1,0) \wedge (0, 3^{-1/2})$.

$N(\mc D_6) = \mat{1&0\\1&0\\1&1\\0&1}$, 
so $N^\top N = \mat{3&1\\1&2}$ and $\mu = \frac 1 2(5 + \sqrt 5)$. 
In general, $N$ is as predicted by the flat structure, 
and $N^\top N$ has $(3, 2, \ldots, 2)$ along the diagonal and $1$ along 
the super- and subdiagonals:
$ \mat{3 & 1 & 0 & \cdots &0 & 0 \\
	1 & 2 & 1 & \cdots &0 & 0 \\
	0 & 1 & 2 & \cdots &0 & 0\\ 
          \vdots&\vdots&\vdots&\ddots&\vdots&\vdots\\
             0 & 0 & 0 & \cdots & 2&	1\\
             0 & 0 & 0 & \cdots & 1&	2\\}$.
The $\mc D_{2n}$ groups can trivially be seen to not have commensurable 
elements with elements of the $\mc A_j$ groups because the former 
exist on a punctured surface. If punctures are added to the $\mc A_j$ 
surfaces, other invariants must be used. 

\newpage
Shown below are configurations for the intersection graph $\mc D_{2n+1}$, 
$n=2,3$. 
The singularity data are $\delta = e_{4n-2}$, 
the same as that for $G(\mc A_{2n})$.
Again, the $G(\mc D_{2n+1})$ for different $n$ have no 
commensurable pseudo-Anosovs, and they also have no pseudo-Anosovs 
commensurable with elements of the $G(\mc A_{2n+1})$.

\bigskip
\begin{figure}[H]
{
\includegraphics{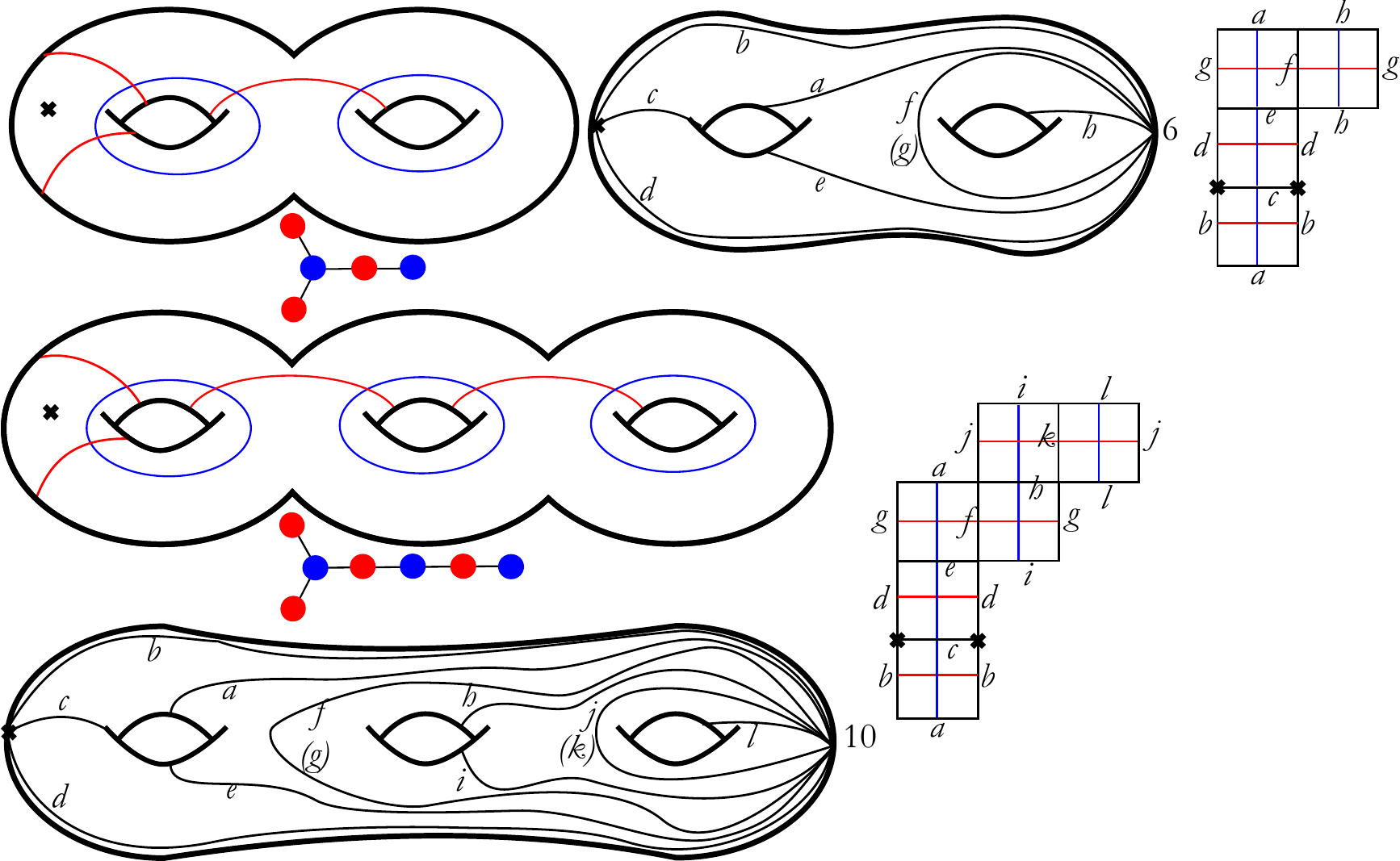}
}
\caption{
$\mc D_{2n+1}$
}
\end{figure}
$N(\mc D_5)^\top = \mat{1 & 1 & 1\\0 & 0 & 1}$, 
so $N^\top N = \mat{3 & 1 \\ 1 & 1}$ and $\mu(\mc D_5) = 2 + \sqrt 2 = \mu(\mc A_7)$.
$\mu(\mc D_7) = 2 + \sqrt 3$.
$N$ is as predicted by the flat structure, 
and $N^\top N$ has $(3, 2, \ldots, 2, 1)$ along the diagonal and $1$ along 
the super- and subdiagonals:
$ \mat{3 & 1 & 0 & \cdots &0 & 0 \\
	1 & 2 & 1 & \cdots &0 & 0 \\
	0 & 1 & 2 & \cdots &0 & 0\\ 
          \vdots&\vdots&\vdots&\ddots&\vdots&\vdots\\
             0 & 0 & 0 & \cdots & 2&	1\\
             0 & 0 & 0 & \cdots & 1&	1\\}$.

\newpage
Our next graphs are $\mc E_{2n}$, $n \geq 3$.
Shown are the cases $n=3,4,5$. 
Since $a \cup b$ does not separate, 
$\delta = e_{4n-2}$, and the different $G(\mc E_{2n})$ 
have no commensurable pseudo-Anosovs.
$G(\mc E_6)$ and $G(\mc E_{8})$ are not free; all others are.

\begin{figure}[p]
{
\includegraphics{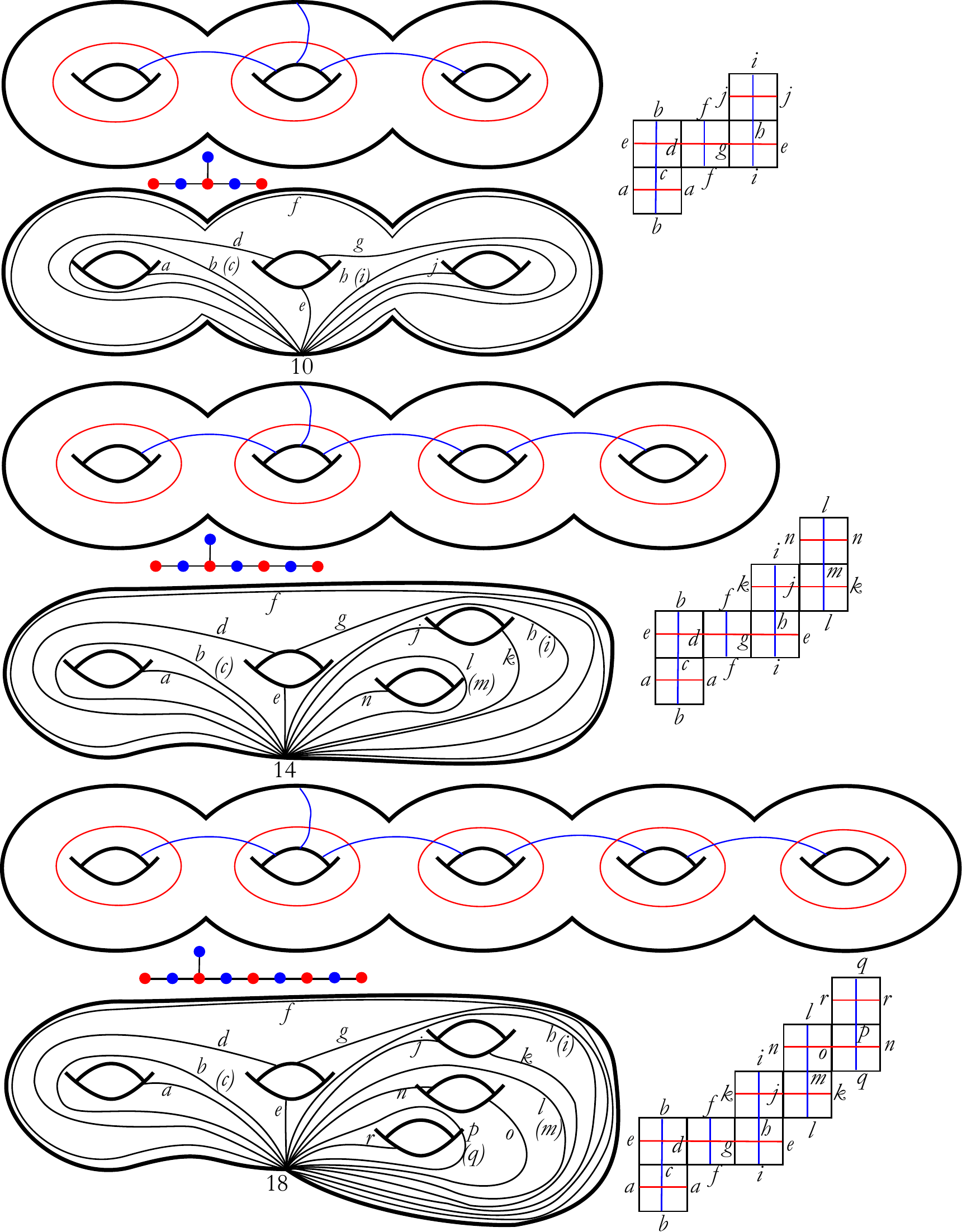}
}
\caption{
$\mc E_{2n}$
}
\end{figure}

\nd $\mu(\mc E_6) = 2 + \sqrt 3$ and in general $NN^\top$
has $(1, 3, 2, \ldots, 2, 1)$ along the diagonal and $1$ along 
the super- and subdiagonals:
$ \mat{1 & 1 & 0 & \cdots &0 & 0 \\
	   1 & 3 & 1 & \cdots &0 & 0 \\
	   0 & 1 & 2 & \cdots &0 & 0\\ 
            \vdots&\vdots&\vdots&\ddots&\vdots&\vdots\\
             0 & 0 & 0 & \cdots & 2&	1\\
             0 & 0 & 0 & \cdots & 1&	1\\}$.
$T_a T_b$ for $\mc E_{10}$ 
turns out to have the smallest $\l$ among all two-multitwist pseudo-Anosovs, 
and this is Lehmer's number.

\newpage
Consider the graphs $\mc E_{2n+1}$, $n \geq 3$. 
Shown is $\mc E_7$. 
We have $\delta = e_{2(n-1)} + e_{2(n+1)}$, so and the different $G(\mc E_{n})$ 
have no commensurable pseudo-Anosovs.
$G(\mc E_7)$ is not free; 
higher $G(\mc E_{2n+1})$ are.

\bigskip
\begin{figure}[H]
{
\includegraphics{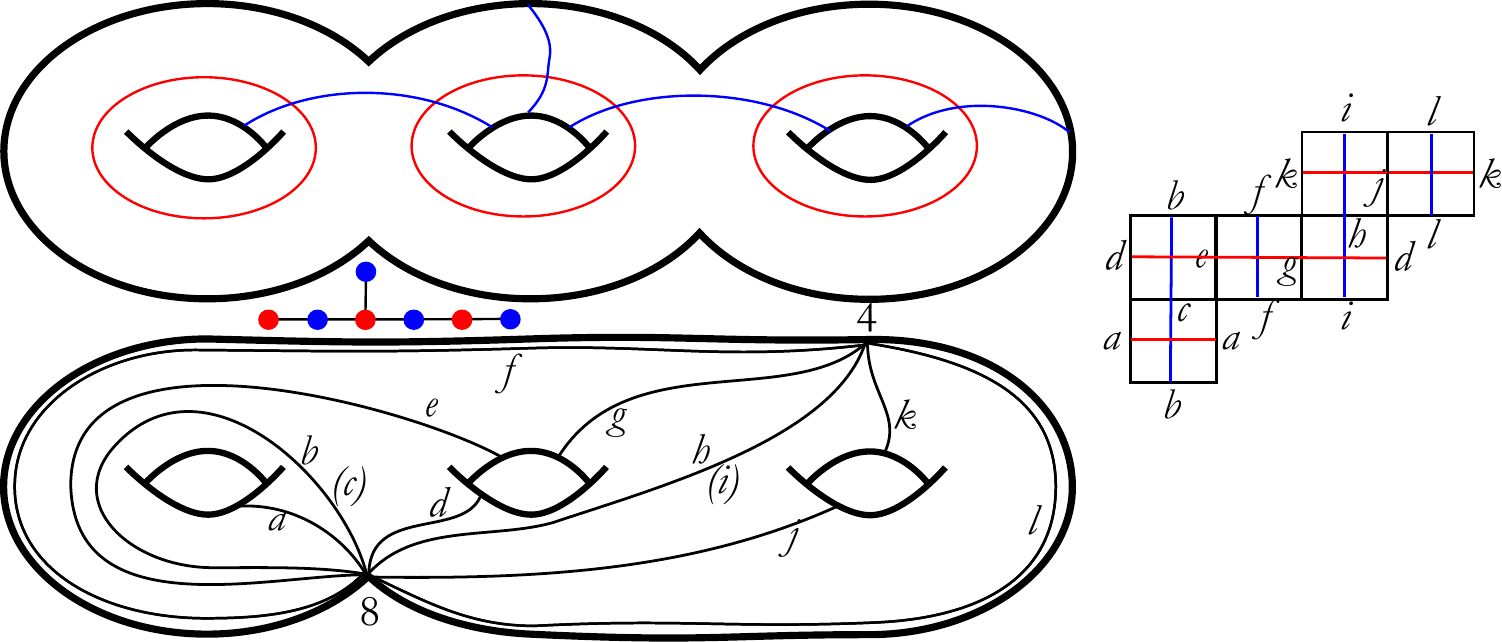}
}
\caption{
$\mc E_{2n+1}$
}
\end{figure}

$\mu(\mc E_7)$ is a root of $x^3 - 6x^2 + 9x - 3$,
and in general $NN^\top$
has $(1, 3, 2, \ldots, 2)$ along the diagonal and $1$ along 
the super- and subdiagonals:
$ \mat{1 & 1 & 0 & \cdots &0 & 0 \\
	   1 & 3 & 1 & \cdots &0 & 0 \\
	   0 & 1 & 2 & \cdots &0 & 0\\ 
            \vdots&\vdots&\vdots&\ddots&\vdots&\vdots\\
             0 & 0 & 0 & \cdots & 2&	1\\
             0 & 0 & 0 & \cdots & 1&	2\\}$.

\newpage
These graphs are called $\mc P_{2n}$. 
Shown are the cases $n=1,2,3,4$. 
We have $\delta = 0$ for the first two, and $\delta = 4e_{n}$ 
for $n \geq 3$.
Thus none of the different groups have commensurable pseudo-Anosovs.

\bigskip

\begin{figure}[H]
{
\includegraphics{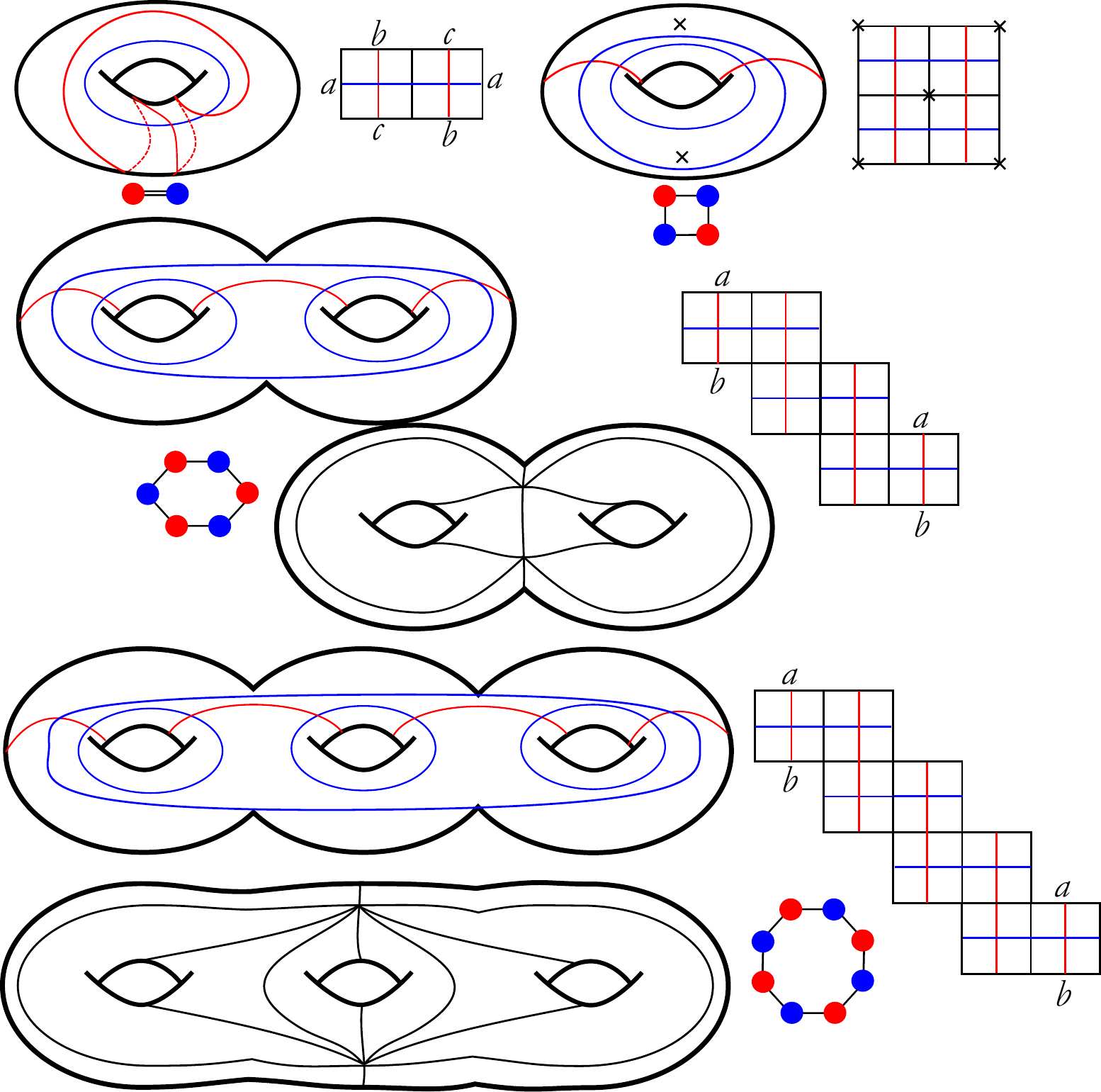}
}
\caption{
$\mc P_{2n}$
}
\end{figure}

For $n = 1$ we have $N = (2)$, so $NN^\top = (4)$ and $\mu = 4$. 
$v = (1)$ is an eigenvector, so $v' = \mu^{-1/2} N^\top v = \frac 1 2 (2)(1) = (1)$ 
is an eigenvector for $N^\top N$, and $J = 4(1,0) \wedge (0,1)$.

For $n=2$ we have $N = \mat{1 & 1 \\ 1 & 1}$, 
so $NN^\top = \mat{2 & 2 \\ 2 & 2}$. Thus again $\mu = 4$. 
$v = v' = \mat{1\\1}$ is an eigenvector, so 
$J = 8 (1,0) \wedge (0,1)$. 

For higher $n$ we have $N = \mat{1 & 1 & 0 & \cdots & 0\\
						      0 & 1 & 1 & \cdots & 0\\
						      0 & 0 & 1 & \cdots & 0\\
						      \vdots & \vdots & \vdots & \ddots & \vdots \\
						      1 & 0 & 0 & \cdots & 1}$, 
 so that
				$NN^\top =  \mat{2 & 1 & 0  & 0 & \cdots & 0 & 1\\
						      1 & 2 & 1 & 0 & \cdots & 0 & 0\\
						      0 & 1 & 2 & 1 & \cdots & 0 & 0\\
			\vdots & \vdots & \vdots & \vdots & \ddots & \vdots & \vdots \\
						      1 & 0 & 0 & 0 & \cdots & 1 & 2}$. 
In all cases, $\mu = 4$, with eigenvector $v = v' = \mat{1 \\ \vdots \\ 1}$,
so $J = 4n (1,0) \wedge (0,1)$.

These graphs are called $Q_{2n+1}$, $n \geq 2$. 
Shown are the cases $n = 2,3,4$. 
They have $\delta = 2 e_{2n-2}$ and $\mu = 4$.

\bigskip

\begin{figure}[H]
{
\includegraphics{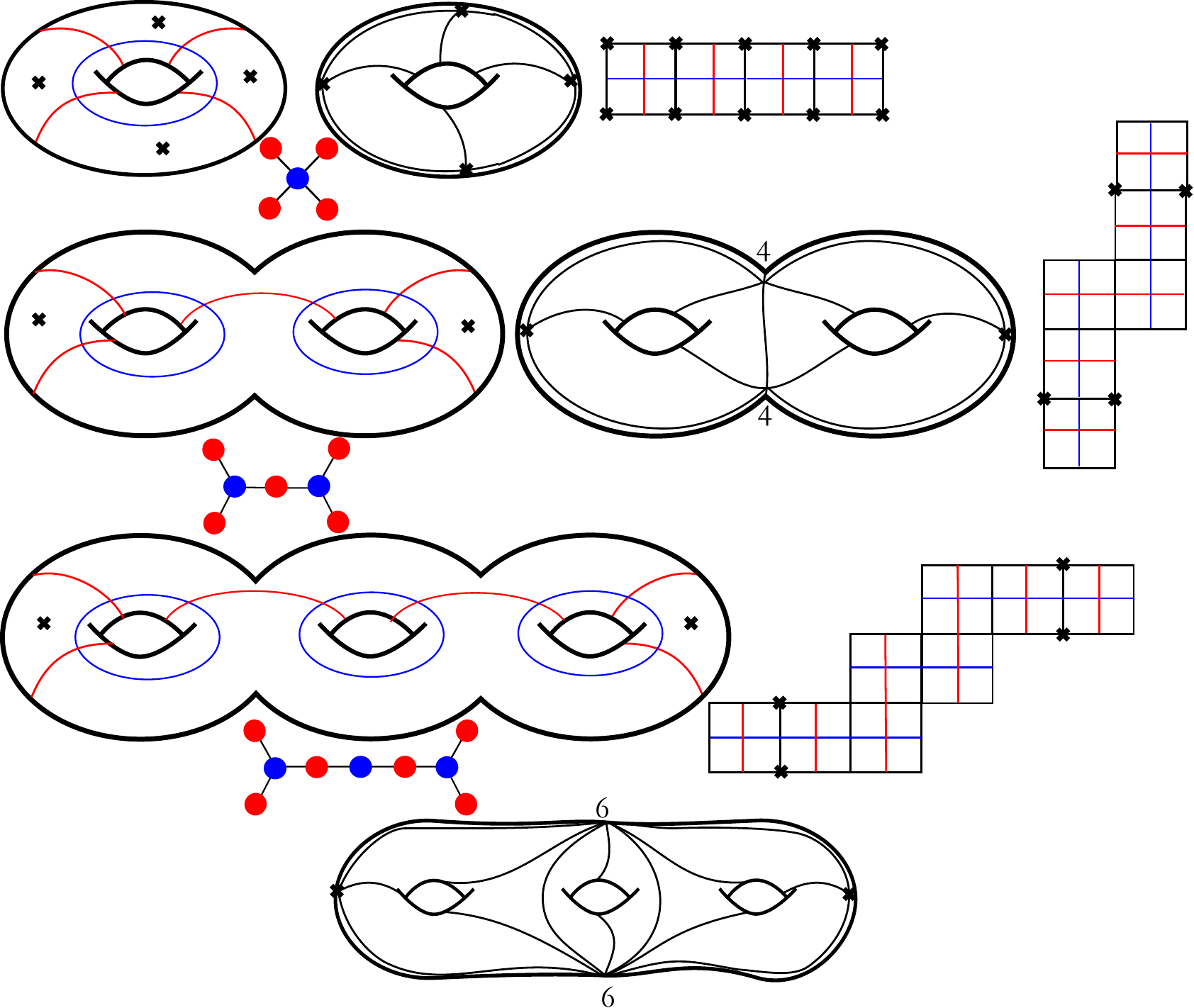}
}
\caption{
$\mc Q_{2n+1}$
}
\end{figure}

\newpage

These graphs are $Q_{2n}$, $n \geq 3$. 
Shown are the cases $n = 3,4$. 
They have $\delta = 2 e_{2n-3}$ and $\mu = 4$.

\bigskip
\begin{figure}
{
\includegraphics{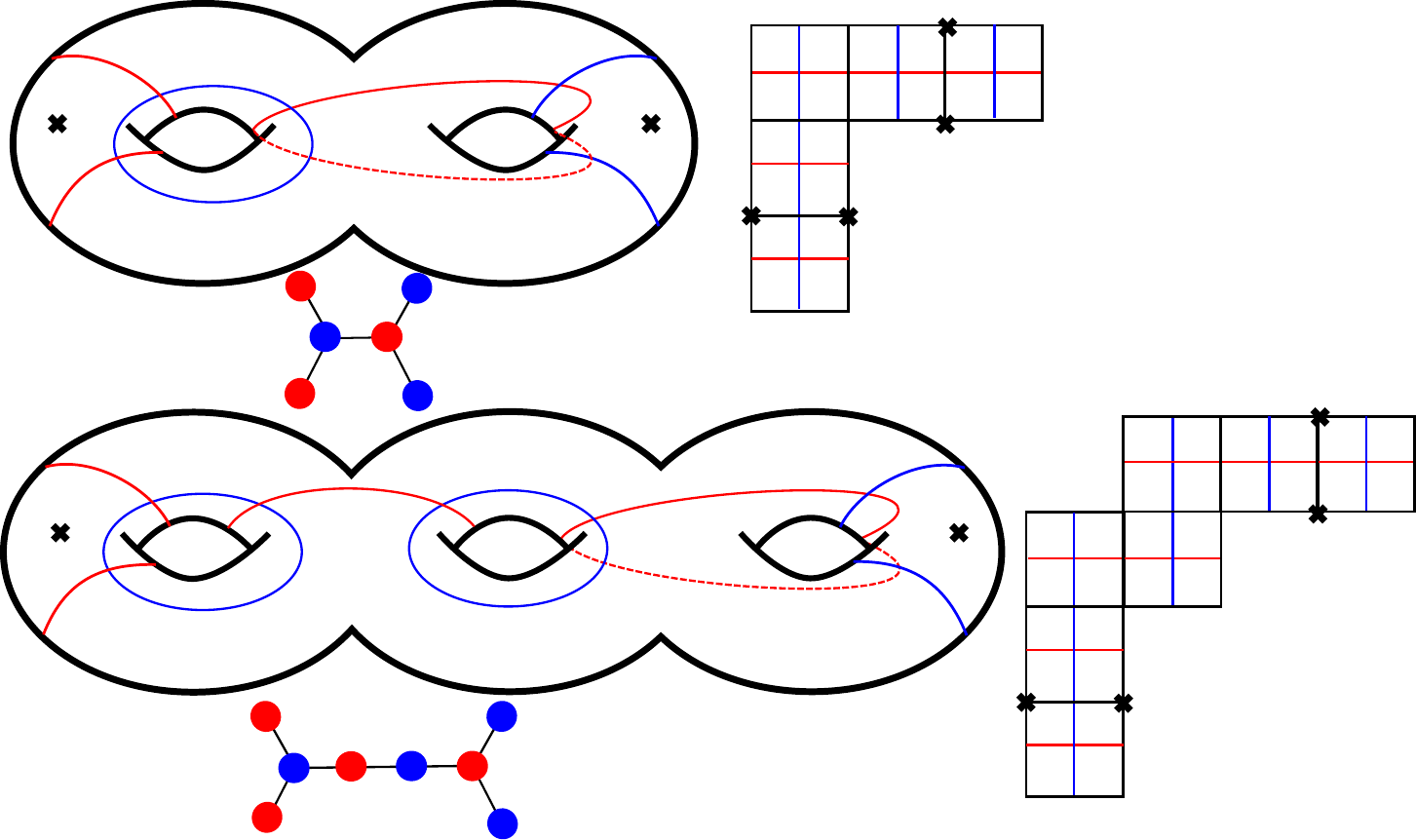}
}
\caption{
$\mc Q_{2n}$
}
\end{figure}

\newpage

Finally, the three graphs $\mc R_7, \mc R_8, \mc R_9$ are below. 
They are the only other connected graphs with $\mu = 4$. 
We have $\delta = 2e_6$, $2e_4 + e_6$, and $e_6 + e_{10}$ respectively 
in the three cases.

\bigskip
\begin{figure}[H]
{
\includegraphics{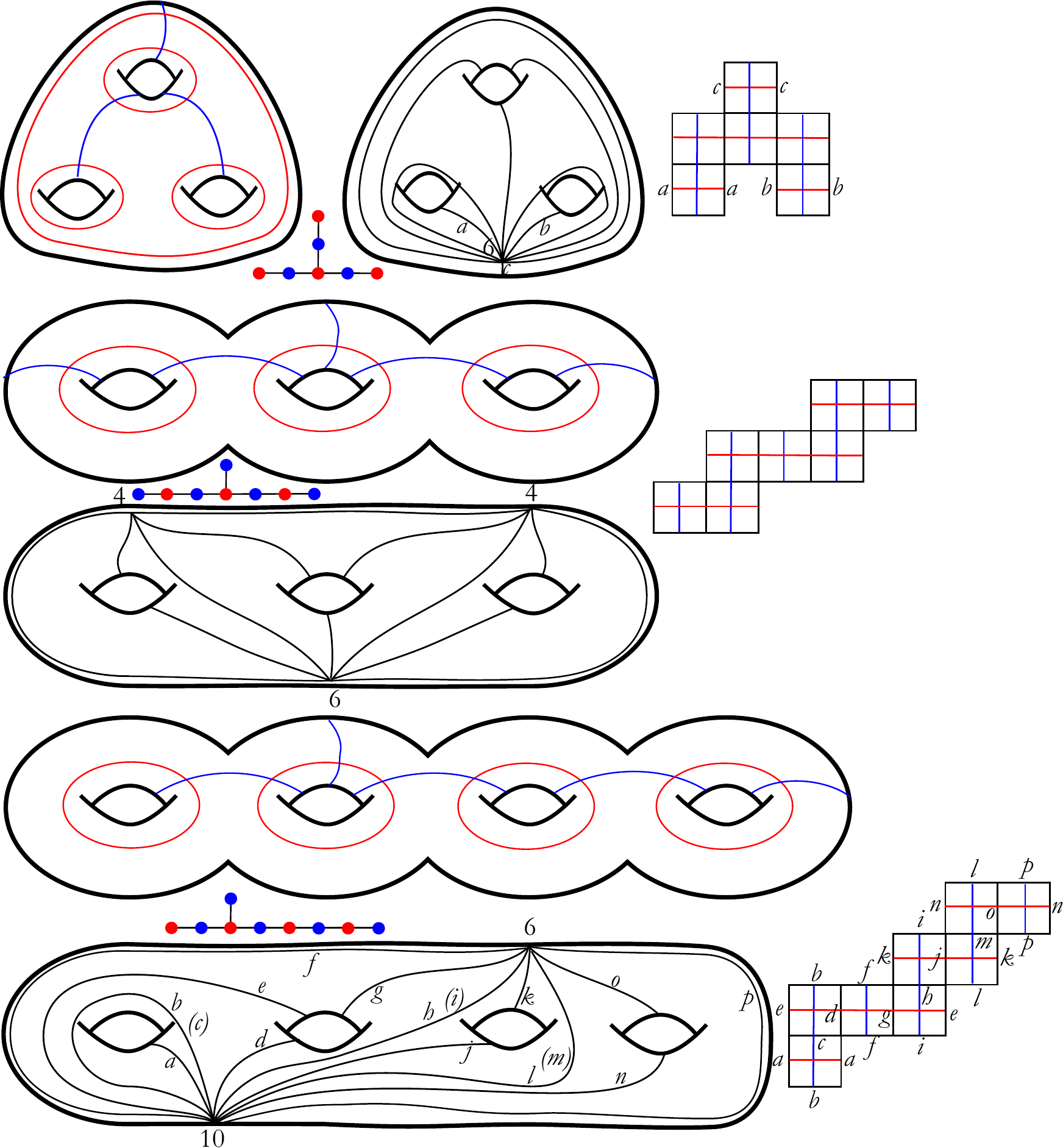}
}
\caption{
$\mc R_7,\mc R_8,\mc R_9$
}
\end{figure}

\newpage
Here are some more multicurve configurations, 
which have no canonical names.

\bigskip
\begin{figure}[H]
{
\includegraphics{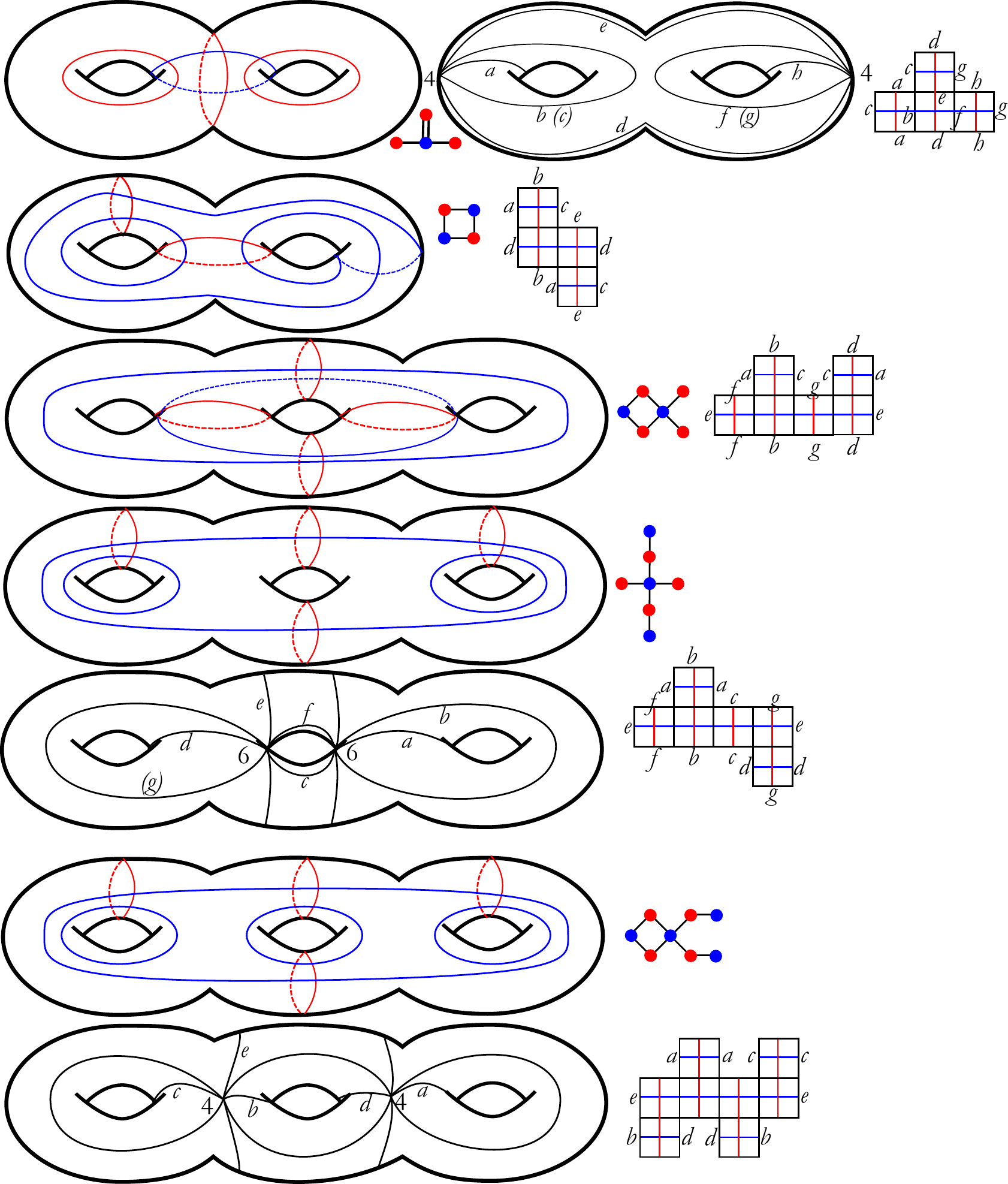}
}
\caption{Miscellaneous curve configurations}
\end{figure}

The first has $\delta = 2e_4$ and $\mu = 6$. 
The second has $\delta = 2e_4$ and $\mu = 4$.
The third has $\delta = 2e_6$ and $\mu = 3 + \sqrt 5$. 
It double covers a genus-two surface with an $\mc A_4$ configuration, 
and we shall meet (a configuration homeomorphic to) it again below.
The fourth has $\delta = 2e_6$ and $\mu = \frac 1 2 (5 + \sqrt{17})$, 
and again double covers a genus-two surface with an $\mc A_4$ configuration.
The last has $\delta = 4e_4$ and $\mu = \frac 1 2(7 + \sqrt{17})$.
It double covers an $\mc A_5$ configuration on a genus-two surface.

\newpage
These miscellaneous graphs have no names.
Their $\delta$s are respectively
$e_{14}$,  \quad
$e_{14}$,  \quad 
$e_6 + e_{10}$,  \quad and
$3e_6$.  

\bigskip
\begin{figure}[H]
{
\includegraphics{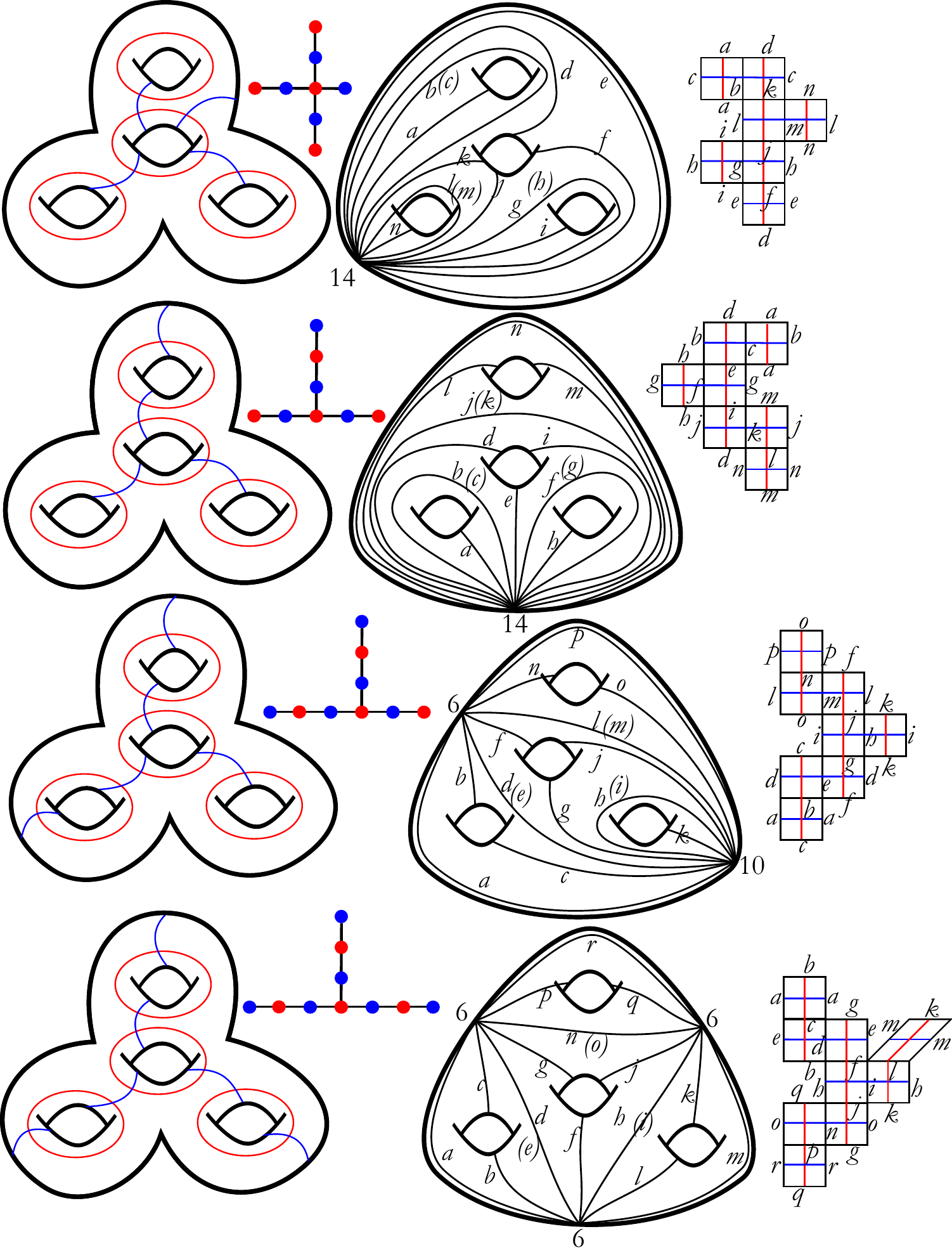}
}
\caption{More miscellaneous curve configurations}
\end{figure}

\nd Their $\mu$ are 
$\frac 1 2(5 + \sqrt{21}), \quad$ 
the greatest root of $x^3 - 6x^2 + 8x - 1, \quad$
the greatest root of $x^3 - 6x^2 + 8x - 2, \quad$ and
$\frac 1 2(5 + \sqrt{13})$.
\bigskip

\newpage
Here are some more miscellaneous graphs that don't have specific names.
Their singularity data are respectively
$2e_3 + 2e_5$,  \quad
$2e_4 + 2e_6$, \quad
$4e_3 + 2e_6$,  \quad
$e_{14}$.

\bigskip
\begin{figure}[H]
{
\includegraphics{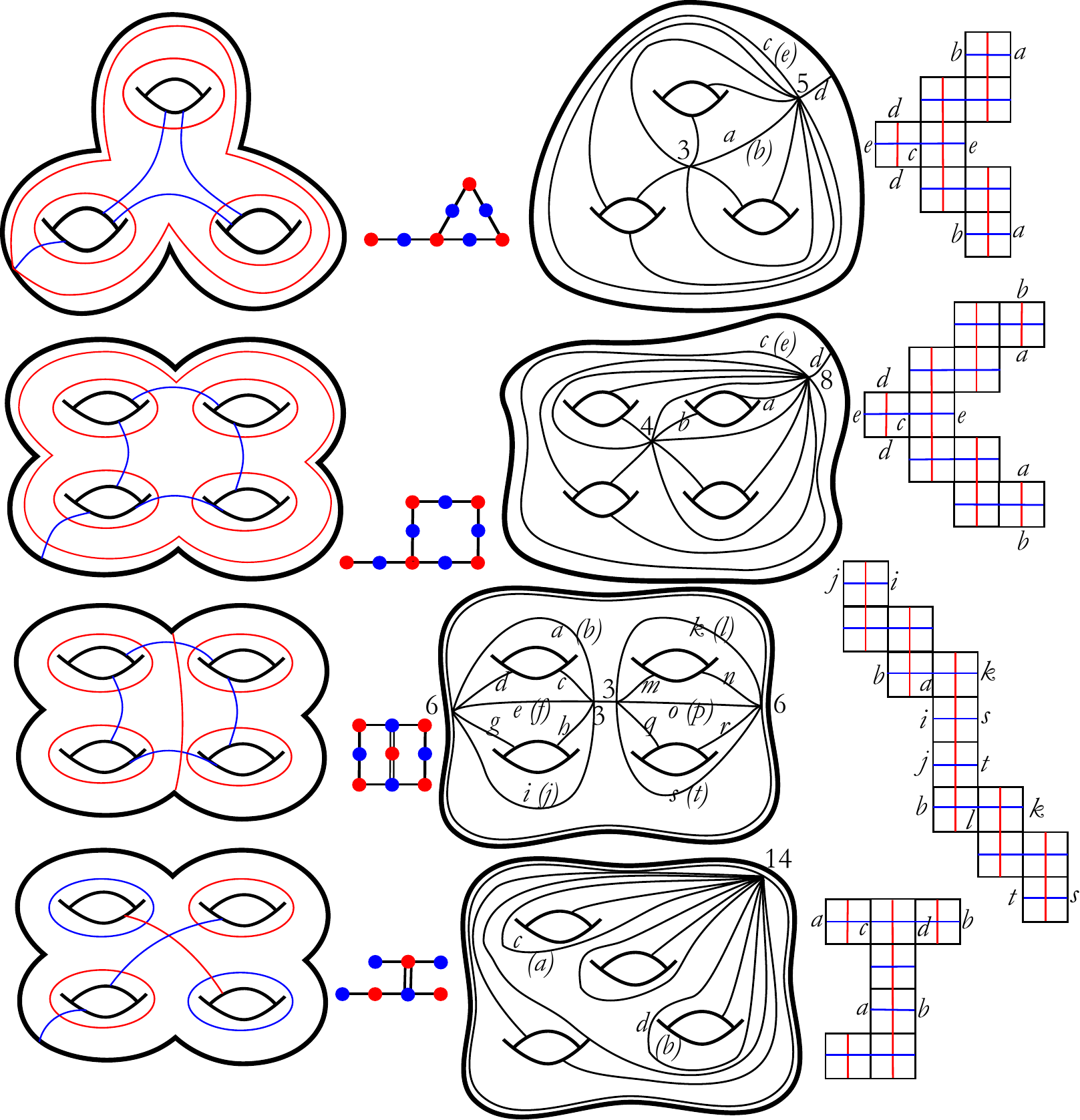}
}
\caption{Even more miscellaneous curve configurations}
\end{figure}
Their $\mu$ are
$\frac 1 2(5 + \sqrt{17})$, \quad
the greatest root of $x^3 - 8x^2 + 18x - 10$, \quad
$6 + \sqrt{20}$, \quad and 
the largest root of $x^3 - 9x^2 + 11x - 2$.

\newpage

A nice class of examples is given by pairs of curves that fill a surface. 
The first three are part of an infinite sequence, and the last is the beginning of 
a different infinite family.

\begin{figure}[H]
\includegraphics{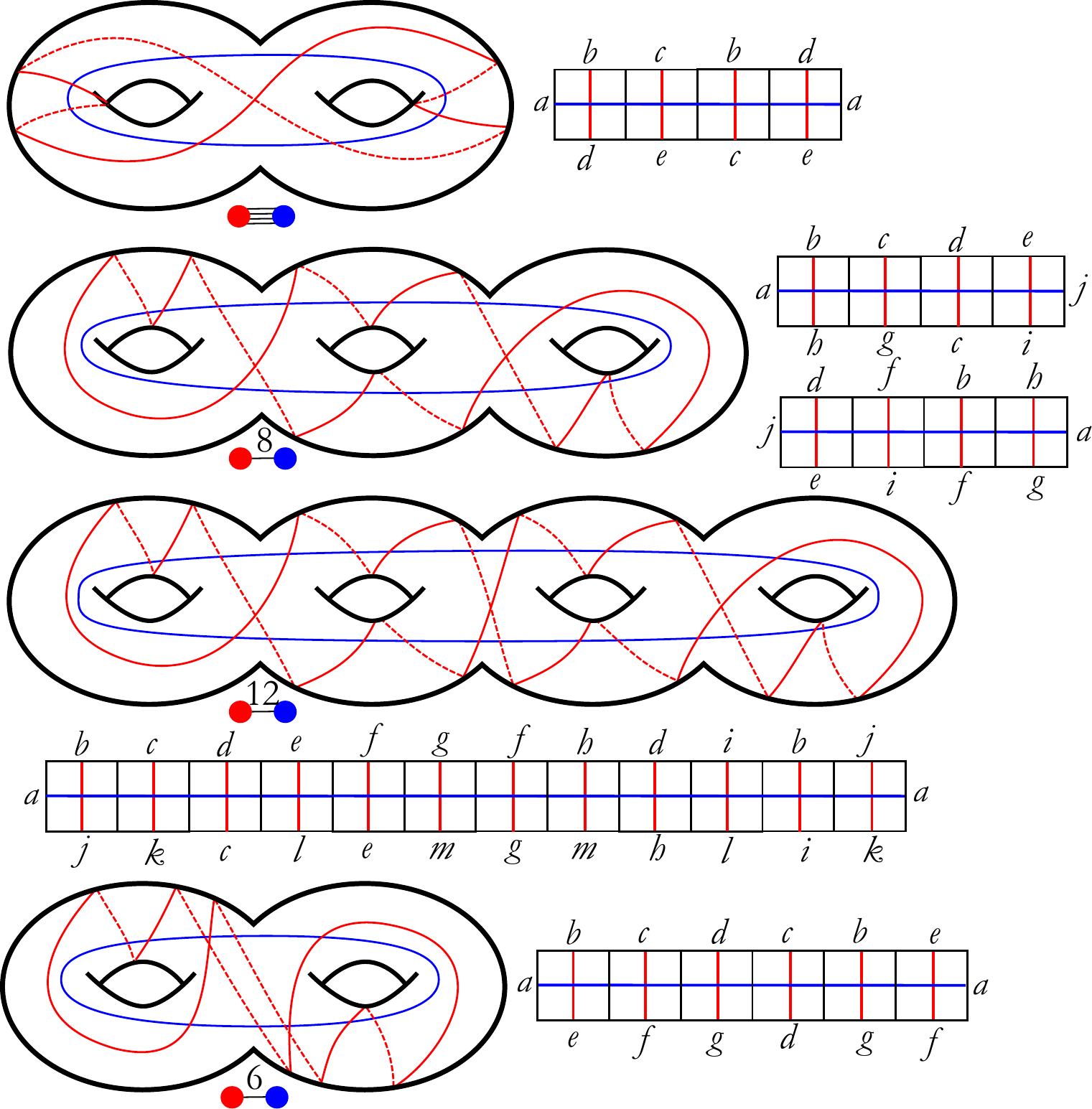}
\caption{pairs of filling curves}
\end{figure}

Their $N$ are 1-by-1 matrices $(n)$ with $n$ the number of intersection points: 
$n = 4(\mbox{genus}-1)$ for the infinite series, and $n=6$ for the last one.
Thus for the sequence, $NN^T = (n^2) = ([4(\mbox{genus}-1)]^2)$, 
which of course has eigenvalue $\mu = n^2$, 
and for the last one, $\mu = 36$. 
They all have eigenvector $v = (1)$, so that $v' = \mu^{-1/2} N^T v 
= \frac 1 n (n)(1) = 1$, and $J = 2n(1,0) \wedge (0,1)$. 

One can calculate their $\delta$ by rotating around vertices in the flat structure
while counting angles, or by reasoning as follows. 
Each square in the flat structure has $2\pi$ of angle inside it,
so there is $4 \cdot 2\pi = 8\pi$ of total angle around vertices in the first one.
A rotation of $\pi$ exchanges the components of the cut surface, 
so there are two vertices each with angle $4\pi$, and $\delta = 2e_4$.
For the rest, one can notice that cutting along the curves results 
in two pieces at the ends, homeomorphic to the ones before, 
and $2(g-2)$ pairwise equivalent components in between.
Since there is $4(g-1) \cdot 2\pi = 8(g-1)\pi$ total angle 
around all vertices, removing the two end pieces leaves $8(g-2)\pi$ 
of angle to be equally divided among $2(g-2)$ vertices, 
meaning each has angle $4\pi$, and $\delta = 2(g-1)e_4$. 
The last example, on cutting, falls into four components, 
one, front and center, 
visibly a square (giving rise to a vertex of angle $2\pi$), 
and another, approximately behind it, a rectangle.
Thus the remaining two vertices have combined angle 
$6 \cdot 2\pi - 2 \cdot 2\pi = 4 \cdot 2\pi$, 
so by symmetry, each has angle $4\pi$, and $\delta = 2e_2 + 2e_4 = 2e_4$.

\newpage
Here is a different construction of two curves filling a surface, 
for closed surfaces of genus $\geq 3$,
generalizing the six-intersection pair on the last page. 
One takes genus-many of the one handled objects below 
and glues them to the singular surface next to them to produce 
a surface with two curves, illustrated to the right in the genus-three case. 
They have $\delta = \mbox{genus }e_4 + 2 e_{\mr{genus}}$ 
and $\mu = (3\cdot \mbox{genus})^2$.

\bigskip
\begin{figure}[H]
\includegraphics{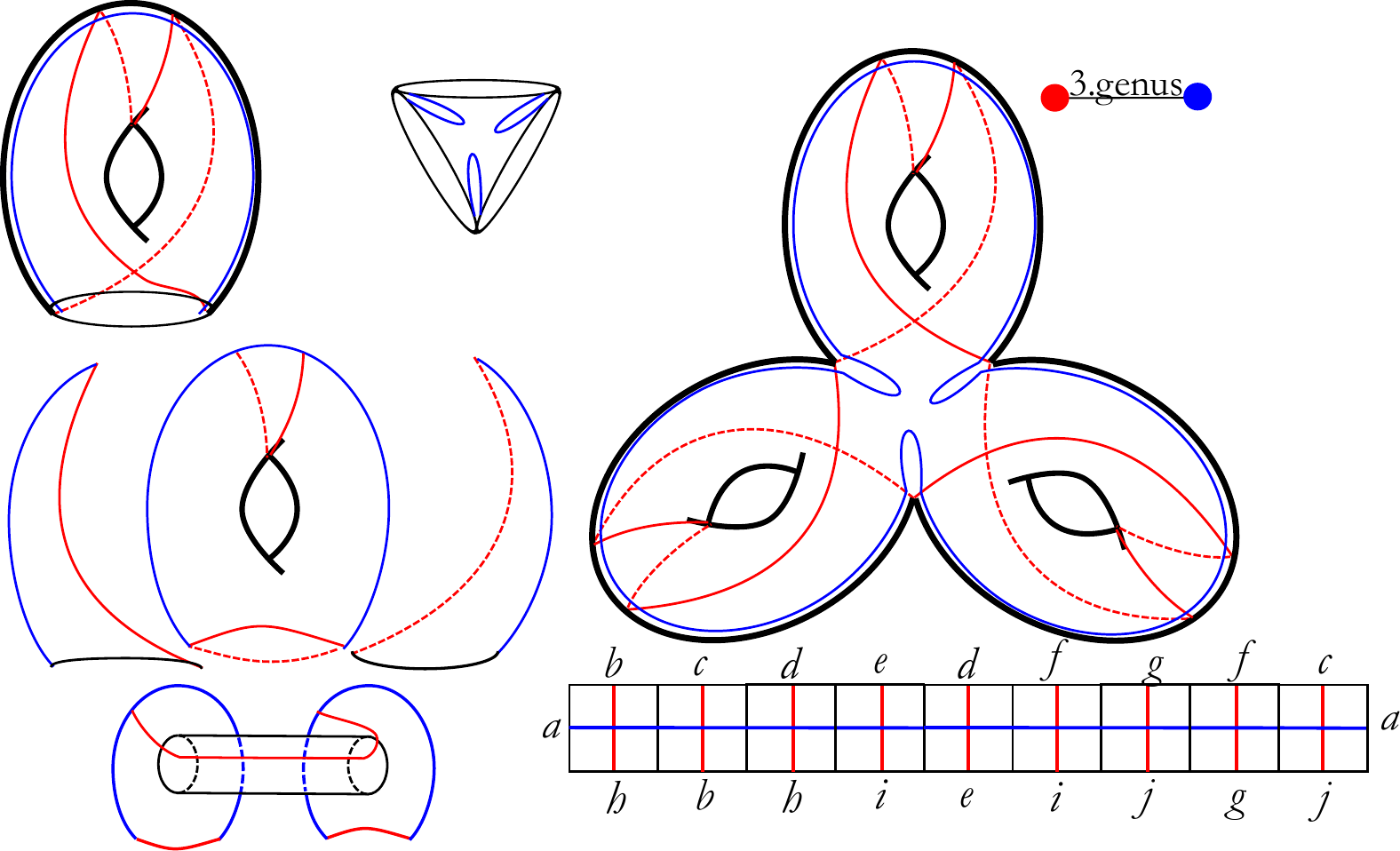}
\caption{More pairs of filling curves}
\end{figure}
\bigskip

Consider the one-handled surface. On cutting along the blue arc, 
the red arc is divided into three segments, whose endpoints along 
what was the blue arc can be moved independently. 
Thus, on cutting along the red arcs, one gets three surfaces, 
shown below the one handled surface, all of them topological disks.
The middle one becomes a cylinder on cutting the blue curve, and 
cutting the remaining red arc makes it a disc.

To see that the glued together surface with the two curves is filled by the curves, 
again one first cuts along the blue curve. The red arcs' endpoints 
on the boundary of the picture are now free to move around, 
so one can move them to the configuration below. 

\bigskip
\begin{figure}[H]
\includegraphics{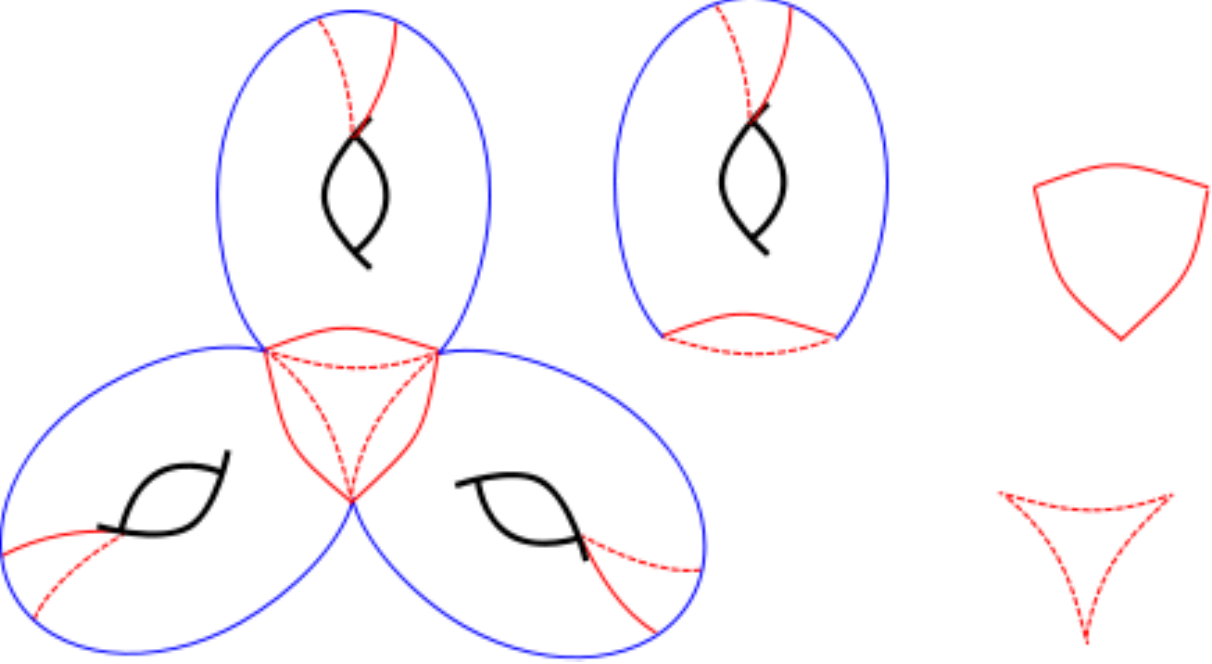}
\caption{Showing the curves fill}
\end{figure}
\bigskip

Now cutting along the red arc gives two triangular discs in the middle 
and genus-many ``handle''-shaped objects, 
which we saw before are actually disks.

\section{A construction of some commensurable pseudo-Anosovs}
\label{construction}

The techniques of the last section allow us to conclude that several accessible 
groups contain no commensurability classes of pseudo-Anosovs in common. 
We are able, however, to construct some commensurable pseudo-Anosovs 
in different $G(a,b)$. Basically, the strategy is to lift Dehn twists 
and use that $p_\#$ is a homomorphism.
We first adopt a strategy that yields compositions of twists 
covering other compositions of twists; 
but these will not lie in two-multitwist groups $G(a,b)$.
On seeing why this strategy does not yield a covering of elements of one $G(a,b)$ 
group by another, a modification that does will become apparent.

In the picture below, we have a cyclic double cover $p\colon \wt S_2 \to S$; 
the covering symmetry 
rotates $\wt S_2$ counterclockwise by an angle of $\pi$ around the axis 
extending vertically through the middle handle. 
Put another way, to get $\wt S_2$ from $S$, cut along $z$, clone the resulting 
doubly-punctured torus, 
and glue the two cut objects together along their boundaries in such a way that 
the resulting action of $\mathbb{Z}/2\mathbb{Z}$ is free.. 

\begin{figure}[H]
\includegraphics{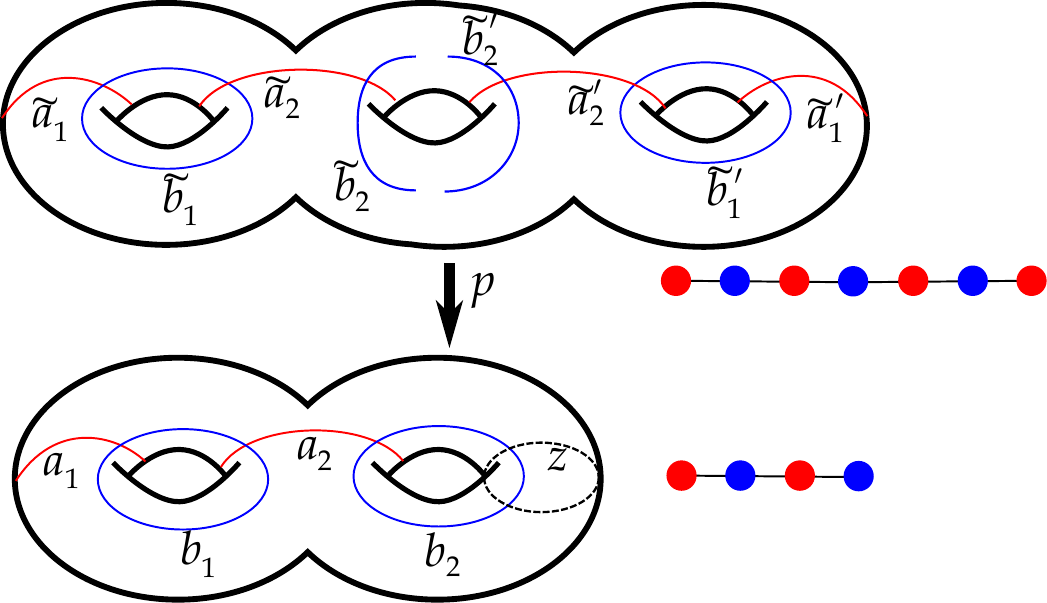}
\caption{A two-fold covering}
\end{figure}

We draw multicurves $a,b$ in an $\mc A_4$ configuration on $S$.
Note that the curves $a_1$, $a_2$, and $b_1$ lift to two pairs of curves 
each on $\wt S_2$, but $b_2$ lifts to a pair of paths since it is cut by $z$.
Let $\wt c := \wt b_2 \wt b_2'$ be the concatenation of these two paths, 
and let the other lifts be as labeled. 
Write $\wt a = \{\wt a_1, \wt a'_1, \wt a_2, \wt a'_2\}$ 
and $\wt b = \{\wt b_1, \wt b'_1, \wt c\}$.
This yields an $\mc A_7$ configuration on $\wt S_2$. 
$z$ is the same curve that in the earlier pictures of the $\mc A_4$ cell decomposition
was labeled $z$. If one cuts that picture along $z$, clones, and glues, 
one gets the earlier picture of the $\mc A_7$ decomposition,
the curves corresponding to the cut curve $z$ there 
being labeled $\wt z$ and $\wt z'$. 

The lengths of curves the corresponding flat structures, however, are not the same, 
(this shows up in the noncommensurability of the different $J$)
and we have seen therefore that none of the pseudo-Anosovs in $G(a,b)$ lift
to elements of $G(\wt a,\wt b)$. 
We can find some compositions of Dehn twists in curves of $a$ and $b$ 
that do lift to $G(\wt a,\wt b)$, but these will not be elements of $G(a,b)$. 
Later, we will find a different cover and multicurves $\wt a,\wt b$ 
such that elements of $G(a,b)$ do lift to elements 
of $G(\wt a, \wt b)$. 
As far as the present cover is concerned,
we claim the Dehn multitwist $T_{\wt a_1} T_{\wt a'_1}$ covers the twist $T_{a_1}$. 

\begin{lemma}
Given an $n$-fold covering of surfaces $p \colon \wt S \to S$, 
if a curve $a$ in $S$ lifts to $n$ curves $\wt a^{(1)},\ldots, \wt a^{(n)}$ in $\wt S$, 
then $p\colon \prod_{j=1}^n T_{\wt a^{(j)}} \to T_a$.
\end{lemma}

\begin{proof}
Consider a small annular neighborhood $A(a)$ of $a$, 
and give $A(a) \homeo [0,1] \x S^1$ coordinates $(t,\theta \pmod {2\pi})$. 
The inverse image $p^{-1} A$ is a disjoint union of $n$ annuli 
$A(\wt a^{(j)})$, each 
homeomorphic to $A(a)$, around $\wt a^{(j)}$.
We can assume the twist $T_{a}$ to be supported on $A(a)$, 
where can be written in local coordinates  
as $(t,\theta) \mapsto (t, \theta + 2\pi t)$.
We can pull back these coordinates to coordinates 
$(\wt t^{(j)}, \wt \theta^{(j)})$ on $A(\wt a^{(j)})$
and $T_{\wt a^{(j)}}$ can be taken to be supported on these  
annuli with local formulas 
$(\wt t^{(j)},\wt \theta^{(j)}) \mapsto (\wt t^{(j)}, \wt \theta^{(j)} + 2\pi \wt t^{(j)})$.
Now if $(\wt t^{(j)},\wt \theta^{(j)}) \in A(\wt a^{(j)})$ 
we have \[p\prod_k T_{\wt a^{(k)}}(\wt t^{(j)},\wt \theta^{(j)}) = 
p T_{\wt a^{(j)}} (\wt t^{(j)},\wt \theta^{(j)}) = 
p (\wt t^{(j)}, \wt \theta^{(j)} + 2\pi \wt t^{(j)}) = 
(t, \theta + 2\pi t) =
T_{a}(t,\theta) = 
T_{a}p (\wt t^{(j)},\wt \theta^{(j)}),\]
so $p\colon \prod_k T_{\wt a^{(k)}} | A(\wt a^{(j)}) \to T_{a} | A(a).$
On the complement  $p\-(S \less A(a))$ of the lifted annuli, 
the multitwist $\prod_j T_{\wt a^{(j)}}$ is the identity, 
as is $T_a$ on $S \less A(a)$,
so we have
$p \prod_j T_{\wt a^{(j)}} = p = T_{a} p$ there as well.
Thus $\prod_j T_{\wt a^{(j)}}$ covers $T_{a}$.
\end{proof}

The same argument shows $T_{\wt a'_2} T_{\wt a_2}$ covers $T_{a_2}$ 
and  $T_{\wt b'_1} T_{\wt b_1}$ covers $T_{b_1}$. 
Only a little subtler is that $T_{\wt c}$ covers $T_{b_2}^2$. 

\begin{lemma}
Given an $n$-fold covering of surfaces $p \colon \wt S \to S$, 
suppose 
a curve $a$ in $S$ lifts to $m$ curves $\wt a^{(1)},\ldots, \wt a^{(m)}$ in $\wt S$, 
with $p | a^{(j)}$ an $n_j$-to-$1$ map. 
Let $L = \lcm \{n_j : j = 1, \ldots, m\}$.
Then $p\colon \prod_{j=1}^m T_{\wt a^{(j)}}^{L/n_j} \to T_a^L$.
\end{lemma}

\begin{proof}
As before, we can take the maps to be the identity off of preselected annuli. 
Take $A(a)$ to be an annulus around $a$ 
and $A(\wt a^{(j)})$ to be the component of $p\- A(a)$ containing $\wt a^{(j)}$.
Restricting to these annuli, it will be enough to show that 
$p T_{\wt a^{(j)}} = T_a^{n_j} p$ on $A(\wt a^{(j)})$.
If $a$ has coordinates $(t,\theta) \in [0,1] \x [0,2\pi)$ as before,
the lift $A(a^{(j)})$ is $n_j$ rectangular neighborhoods with these coordinates, 
attached end to end, so they 
naturally carry coordinates $(\wt t, \wt \theta) \in [0,1] \x [0,2 n_j\pi)$. 
The point $(t,\theta) \in A(b_2)$ has the $n_j$ lifts 
$(t, \theta + k\pi)$ upstairs, $0 \leq k < n_j$ 
so the $n_j$-fold covering $p | A(\wt a^{(j)})$ can 
be represented by taking the second coordinate modulo $2\pi$.
Now we can take $T_{\wt a^{(j)}}(\wt t, \wt \theta) = (t, \wt \theta + n_j\pi \wt t)$, 
and projecting down, 
\[pT_{\wt a^{(j)}} (\wt t, \wt \theta) = (t, \theta + n_j\pi t \pmod{2\pi})
= T_{a}^{n_j} p(\wt t, \wt \theta).\]
\end{proof} 

Now using that $p_\#$ is a homomorphism, 
we see $p \colon T_{\wt a} \to T_a$ and $p \colon T_{\wt b} \to T_{b_1} T_{b_2}^2$. 
Thus we have an element-by-element covering 
$p\colon G(\wt a,\wt b) \to \ang{T_a, T_{b_1} T_{b_2}^2}$. 
However, the latter group intersects $G(a,b)$ in $\ang{T_a}$, 
so they share no pseudo-Anosovs.

We now see why in this example elements of $G(a,b)$ don't lift to elements of 
$G(\wt a,\wt b)$: the problem is exactly that the exponents of $T_{b_1}$ 
and $T_{b_2}$ are not the same in the image. In order to fix this, 
we should arrange that $b_1$ doesn't lift. 
So consider the two-fold
covering below.

\begin{figure}[H]
\includegraphics{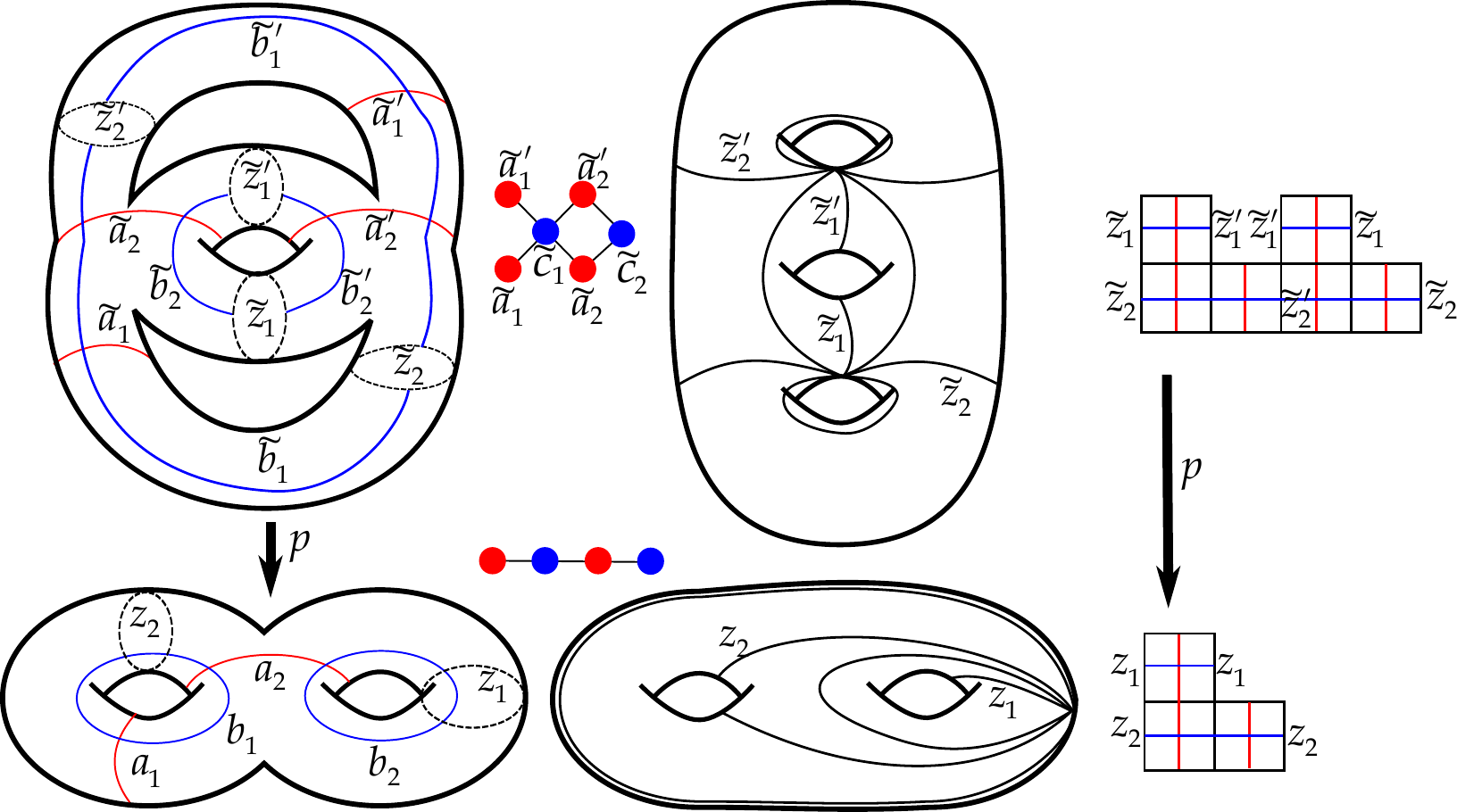}
\caption{A ``better'' two-fold cover}
\end{figure}

The $\mc A_4$ curve pattern downstairs is the same, 
but now we are cutting along a multicurve $z$ that meets both $b_1$ and $b_2$, 
so each is lifted to a pair of paths. Join the lifts of $b_1$ to form a curve $\wt c_1$ 
double-covering $b_1$ and similarly form a curve $\wt c_2$ double-covering $b_2$. 
Now $p_\#(T_{\wt c_j}) = T_{b_j}^2$ for $j = 1,2$, 
so if we let $\wt a$ be the set of lifts of $a_1,a_2$ and $\wt c = \{\wt c_1, \wt c_2\}$, 
then $p_\#(T_{\wt a}) = T_a$ and $p_\#(T_{\wt c}) = T_b^2$, 
so $p_\#\colon G(\wt a, \wt c) \to \ang{T_a,T_b^2} < G(a,b)$.

There is nothing special about double covers in this example; forming an $n$-fold 
cover of $S$ by cutting along $z$ and gluing $n$ copies cyclically gives 
a configuration $\wt a^{(n)}, \wt c^{(n)}$ of curves in the cover $\wt S_n$ such that
$p_\#\colon G(\wt a^{(n)},\wt c^{(n)}) \to \ang{T_a,T_b^n} < G(a,b)$.

In this way we get an infinite family of cyclic covers $p_n\colon \wt S_n \to S$; 
moreover, the cover $\wt S_{mn} \to S$ always factors both through $\wt S_m$ 
and $\wt S_n$: 
\xymatrix{
   & \wt S_{mn} \ar[dl] \ar[dr] & \\
    \wt S_m \ar[dr] & & \wt S_n \ar[dl] \\
    & S. &}\\
The curve systems are compatible with the coverings, so that 
the group $G(\wt a^{(mn)}, \wt c^{(mn)})$ in $\MCG(\wt S_{mn})$
covers the subgroup $\ang{T_{\wt a^{(n)}}, T_{\wt c^{(n)}}^m}$ 
of $G(\wt a^{(n)}, \wt c^{(n)}) < \MCG(\wt S_n)$ 
under the covering map $\wt S_{mn} \to \wt S_n$.
This shows that the elements in the group 
$\ang{T_{\wt a^{(n)}},T_{\wt c^{(n)}} ^m} < \MCG(\wt S_n)$ 
are all commensurable with elements of the group
$\ang{T_{\wt a^{(m)}},T_{\wt c^{(m)}} ^n}$ of $\MCG(\wt S_m)$.

The transpose of the incidence matrix for the upstairs configuration is 
$N^\top = \mat{1 & \cdots & 1 & 1 & \cdots & 1\\
	                  0 & \cdots & 0 & 1 & \cdots & 1}$, 
so that $N^\top N = \mat{2n&n\\n&n} = n \mat{2 & 1\\1&1}$ 
has characteristic polynomial $x^2 - 3nx + n^2$ and Perron--Frobenius eigenvalue 
$n\frac{3 + \sqrt 5}{2} = n\mu$. 
Thus in the derivative representation $G(\wt a^{(n)}, \wt c^{(n)}) \to \PSL(2,\R)$ 
we have $T_{\wt a^{(n)}} \mapsto \mat{1 & 0 \\ -\sqrt{n\mu} & 1}$ 
and
$T_{\wt c^{(n)}} \mapsto \mat{1 & \sqrt{n\mu} \\ 0 & 1}$.

The ``natural choice'' of width assignment for rectangles upstairs is 
to glue together $n$ copies of the flat structure we've chosen for the surface below:
namely, to give lifts of $a_1$ the width $1$ 
and lifts of $a_2$ the width $\g$, the golden ratio, 
$c^{(n)}_1$ the width $\g$, and
$c^{(n)}_2$ the width $1$.
If we take derivatives with respect to this flat structure, though,
we get $T_{\wt a^{(n)}} \mapsto \mat{1 & 0 \\ -\sqrt \mu & 1}$ 
and $T_{\wt c^{(n)}} \mapsto \mat{1 & n\sqrt{\mu} \\ 0 & 1}$, 
which factors through the representation $G(a,b) \to \PSL(2,\R)$, 
but is a different representation from the one in the last paragraph.
This difference demonstrates the ambiguity 
in choosing flat structures we discussed in Section 4. 
While the choice of eigenvectors for $NN^\top$ and $N^\top N$ 
we have made in Section 3 gives the most natural representation, 
it does not factor through covering maps $p_\#$.

It is now clear how to generalize this example.

\begin{theorem}
Let an $n$-fold covering of surfaces $p \colon \wt S \to S$ be given.
Suppose $a$ and $b$ are multicurves filling $S$, 
with components $a_1,\ldots,a_l$ and $b_1,\ldots,b_m$.
Suppose $a_j$ has lifts $\wt a_j^{(k)}$ with $p | \wt a_j^{(k)}$ an $n_{j,k}$-to-$1$ 
map and $b_j$ has lifts $\wt b_j^{(k)}$ with $p | \wt b_j^{(k)}$ an $n'_{j,k}$-to-$1$ 
map. 
Let $L_j = \lcm \{n_{j,k} : k\}$ and $L'_j = \lcm \{ n'_{j,k}: k\}$ for each $j$, 
and further suppose $L = L_j$ and $L' = L'_j$ are independent of $j$.
Then $p\colon \prod_{j,k} T_{\wt a_j^{(k)}}^{L/n_{j,k}} \to T_a^L$ 
and $p\colon \prod_{j,k} T_{\wt b_j^{(k)}}^{L'/n'_{j,k}} \to T_b^{L'}$. 
In particular, if $L = n_{j,k}$ and $L' = n'_{j,k}$ are independent of $j,k$, 
then letting $\wt a = \{\wt a_j^{(k)}\}_{j,k}$ and $\wt b = \{\wt b_j^{(k)}\}_{j,k}$
we have a covering 
$p_\#\colon \ang{T_{\wt a},T_{\wt b}} \to \ang{T_a^L,T_b^{L'}}$.\end{theorem}

\begin{proof}
This follows from restricting to annuli and applying the preceding two lemmas.
\end{proof}

\nd Example. In the picture below are two covers of the $\mc A_4$ configuration.
The one on the right is from before, but the one on the left is new; 
one can get it by cutting along the curves indicated. 
They have a common double cover, a $\mathbb{Z}/2\mathbb{Z} \x \mathbb{Z}/2\mathbb{Z}$ cover of $\mc A_4$.

\begin{figure}[H]
\includegraphics{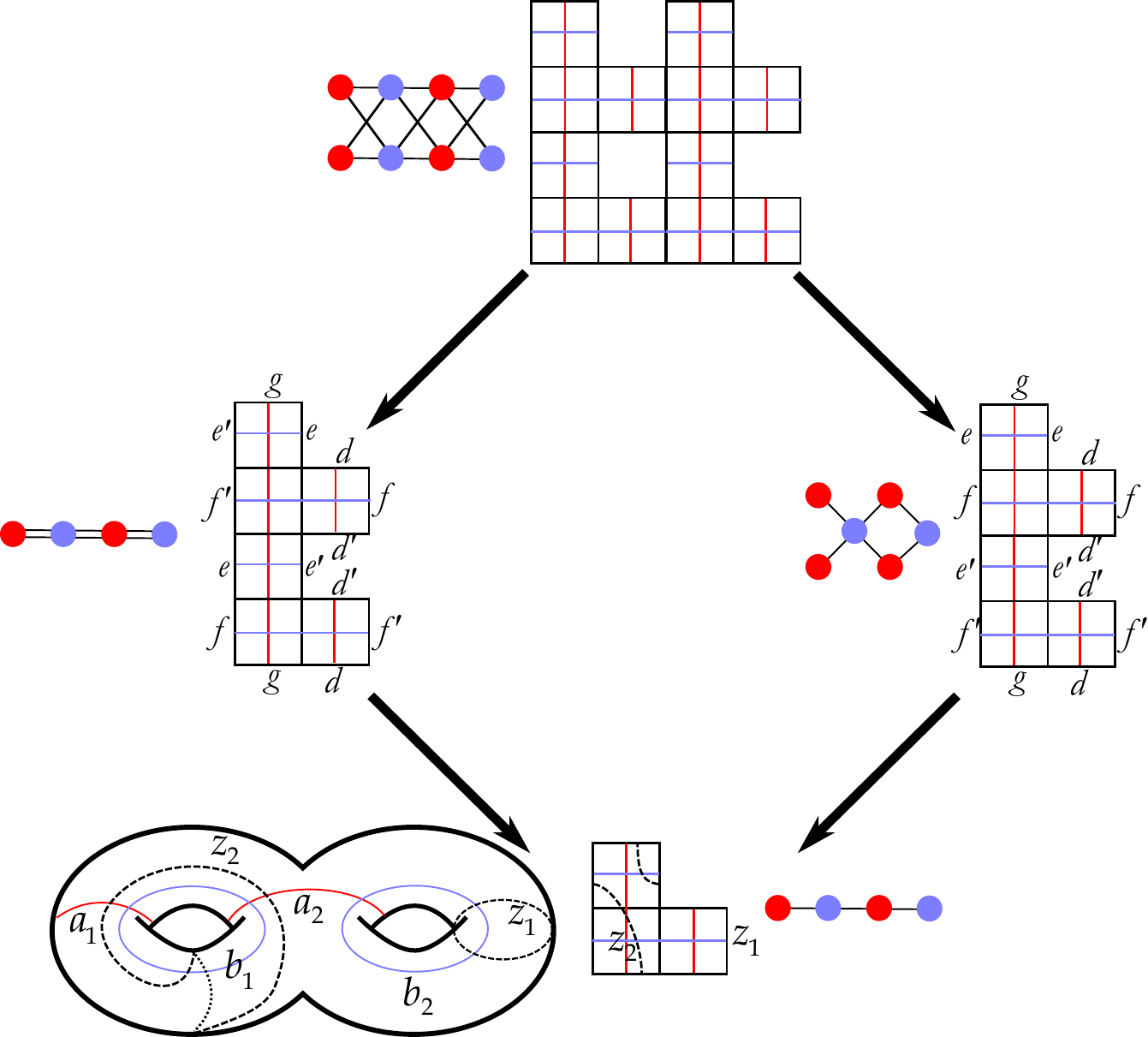}
\caption{A pair of commensurable groups}
\end{figure}
\bigskip

The red curves lift on both sides. 
The blue curves on the left lift, while those on the right are cut when we  
produce the double cover. 
If $a,b$ are the red and blue multicurves on the left and $a',b'$ 
are the red and blue multicurves on the right, respectively, 
then we have that elements of $\ang{T_a,T_b}$ are commensurable 
with elements of $\ang{T_{a'},T_{b'}^2}$, 
in the sense that the homomorphism defined by 
$T_a \mapsto T_{a'}$ and $T_b \mapsto T_{b'}^2$ takes elements to 
commensurable elements.

A little reflection shows that we never have a full element-by-element group-covering 
$G(\wt a,\wt b) \to G(a,b)$ when $G(a,b)$ contains 
pseudo-Anosovs and the covering 
surface is connected; 
the process of forming an $n$-fold covering of a surface $S$ involves cutting along 
some curves (and/or arcs between boundary components) $z$, 
cloning the cut surface $n$ times, 
and gluing the copies of the boundary components together. 
Since the covering surface is assumed connected, the permutation of $\{1,\ldots,n\}$ 
induced by the gluing corresponding to some component $z_1$ of $z$ must 
not be the identity.
If $G(a,b)$ contains pseudo-Anosovs, $a \cup b$ fills $S$, and so at least one curve,
say $a_1$, must intersect $z_1$. 
Then $a_1$ lifts to a collection of paths, which concatenate together to 
a collection $\wt a_1$ of curves, 
at least one of which multiply covers $a_1$. 
If these multiplicities are the same number $m$ for all the curves covering $a_1$, 
then $T_{\wt a_1}$ covers $T_a^m$; 
if the multiplicities are different, no automorphism of the covering surface covers 
a power of $T_{\wt a_1}$. Either way, not all of $G(a,b)$ can be covered.


\newpage



\begin{thebibliography}{9999}

\bibitem{CSW} \label{CSW} D. Calegari and H. Sun and S. Wang, 
On fibered commensurability, 
{arXiv:1003.0411v1} [math.GT], 1 Mar 2010.

\bibitem{Gant} 
F. Gantmacher, {\em The theory of matrices}, vol. 2, Chelsea (1959)

\bibitem{Lein} C. J. Leininger,
On groups generated by two positive multi-twists: 
Teichm\"{u}ller curves and Lehmer's number,
{\em Geometry and Topology} {\bf 8} (2004), 1301--1359

\bibitem{KS} R. Kenyon and J. Smillie, 
Billiards on rational-angled triangles, 
{\em Comment. Math. Helv.} {\bf 75} (2000), 65--108

\bibitem{McM}
C. T. McMullen, 
Billiards and Teichm\"{u}ller curves on Hilbert modular surfaces,
{\em J. Amer. Math. Soc.}  {\bf 16}  (2003), 857--885

\bibitem{Otal}
J.-P. Otal, Le th\'{e}or\`{e}me d'hyperbolisation pour les vari\'{e}t\'{e}s 
fibr\'{e}es de dimension 3, {\em Ast\'{e}risque} 
{\bf 235}, 
Soci\'{e}t\'{e} Math\'{e}matique de France, Paris (1996)

\bibitem{Thur} W. P. Thurston, 
On the geometry and dynamics of diffeomorphisms of surfaces, 
{\em Bull. Am. Math. Soc.} (New Ser.) {\bf 19} (1988) 417--431

\bibitem{Walsh} G. S. Walsh, 
Orbifolds and commensurability,
{arXiv:1003.1335v1} [math.GT],
5 Mar 2010.

\bibitem{Zor}
Anton Zorich,
Flat Surfaces,
{\em Frontiers in Number Theory, Physics, and Geometry} Vol. I, 
P. Cartier; B. Julia; P. Moussa; P. Vanhove (Editors), Springer Verlag (2006),
{arXiv:0609392v2} [Math.DS]


\end{thebibliography}
\end{document}